\documentclass[a4paper,12pt]{article}
\usepackage[utf8]{inputenc}
\usepackage{fullpage}
\usepackage{amsmath, amssymb, amstext, amsfonts, amsthm, bbm, array, enumerate,scalerel}
\usepackage{mathrsfs, dsfont,mathabx}
\usepackage{nicefrac}
\usepackage[none]{hyphenat}
\usepackage[dvipsnames]{xcolor}
\usepackage{tikz}
\usepackage{stmaryrd}
\usepackage{tikz-3dplot}
\usepackage{float}
\usepackage{appendix}
\usepackage{longtable}
\usepackage[round]{natbib}
\usetikzlibrary{patterns}
\usetikzlibrary{plotmarks}

\newtheoremstyle{boldremark}
    {\dimexpr\topsep/2\relax} 
    {\dimexpr\topsep/2\relax} 
    {}          
    {}          
    {\bfseries} 
    {.}         
    {.5em}      
    {}          

\RequirePackage[colorlinks,citecolor=blue,urlcolor=blue, linkcolor=blue]{hyperref}

\theoremstyle{plain}
\newtheorem{theorem}{Theorem}[section]
\newtheorem{lemma}[theorem]{Lemma}
\newtheorem{corollary}[theorem]{Corollary}
\newtheorem{proposition}[theorem]{Proposition}

\theoremstyle{definition}
\newtheorem{definition}[theorem]{Definition}
\newtheorem{example}[theorem]{Example}
\newtheorem{conjecture}[theorem]{Conjecture}

\newtheorem*{assumption*}{\assumptionnumber}
\providecommand{\assumptionnumber}{}
\makeatletter

\theoremstyle{boldremark}
\newtheorem{remark}[theorem]{Remark}

\numberwithin{equation}{section}

\newcommand{\E}[2][]{\mathbb{E}_{#1}\left[#2\right ]}

\renewcommand{\P}[2][]{\mathbb{P}_{#1}\left(#2\right )}
\newcommand{\R}{\mathbb{R}}

\newcommand{\N}{\mathbb{N}}
\newcommand{\ind}[1]{\mathds{1}_{#1}}

\renewcommand{\L}{\Lambda}
\newcommand{\Linfmin}{\underline{\Lambda}}

\newcommand{\EmpM}{\mu}
\newcommand{\EmpMphys}{\hat{\mu}}
\newcommand{\EmpMext}{\xi}
\newcommand{\EmpMextphys}{\hat{\xi}}
\newcommand{\muext}{\xi}
\newcommand{\cadlag}{c\`adl\`ag~}

\newcommand{\Linfphys}{\hat \Lambda}
\newcommand{\Lmin}{\underline{\mathsf{L}}}
\newcommand{\taumin}{\underline{\tau}}
\newcommand{\tauphys}{\hat \tau}
\newcommand{\Xmin}{\underline{X}}
\newcommand{\Xphys}{\hat X}
\newcommand{\Lphys}{\hat \Lambda}

\newcommand{\Lpertmin}{\underline{\tilde{\Lambda}}{}}
\newcommand{\Xpert}{\tilde{X}}
\newcommand{\Xpertmin}{\underline{\tilde{X}}{}}
\newcommand{\EmpMpert}{\underline{\tilde{\mu}}{}}

\newcommand{\taupert}{\tilde{\tau}}

\newcommand{\lawmin}{\underline{\mu}}
\newcommand{\lawphys}{\hat{\mu}}
\newcommand{\nearopt}{\zeta}

\newcommand{\GammaNpert}{\tilde{\Gamma}_N}
\newcommand{\PErandom}{\mathcal{R}}
\newcommand{\DistFunctions}{M}
\newcommand{\Graph}{\mathscr{G}}
\newcommand{\ExtendedE}{\bar{E}}
\newcommand{\ellfunc}{\lambda}

\newcommand{\lawXnull}{\nu_{0-}}
\newcommand{\lawstopped}{\nu}
\newcommand{\lawstoppedmin}{\underline{\nu}}
\newcommand{\lawstoppedphys}{\hat{\nu}{}}

\newcommand{\Levymetric}{d_L}
\newcommand{\e}{\varepsilon}
\newcommand{\ee}{\delta}
\newcommand{\law}{\operatorname{law}}

\newcommand{\ExtendedEproj}{\overline{\pi}}

\newcommand\mydots{\hbox to 1em{.\hss.\hss.}}

\title{Propagation of minimality in the supercooled Stefan problem }
\author{Christa Cuchiero\thanks{Vienna University, Department of Statistics and Operations Research, Data Science @ Uni Vienna, Kolingasse 14-16, A-1090 Wien, Austria, christa.cuchiero@univie.ac.at}
\and Stefan Rigger \thanks{Vienna University, Faculty of Mathematics, Kolingasse 14-16, A-1090 Wien, Austria, stefan.rigger@univie.ac.at.}
\and Sara Svaluto-Ferro\thanks{Vienna University, Faculty of Mathematics, Kolingasse 14-16, A-1090 Wien, Austria, sara.svaluto-ferro@univie.ac.at.\newline
The authors gratefully acknowledge financial support by the Vienna Science and Technology Fund (WWTF) under grant MA16-021. 
}}
\date{}

\begin{document}
\maketitle

\begin{abstract}
Supercooled Stefan problems describe the evolution of the 
boundary between the solid and liquid phases of a substance, where the liquid is assumed to be cooled below its freezing point. 
Following the methodology of Delarue, Nadtochiy and Shkolnikov, 
we construct solutions to the one-phase one-dimensional supercooled Stefan problem through a certain McKean--Vlasov equation, which
allows to define global solutions even in the presence of blow-ups. 
Solutions to the McKean--Vlasov equation arise as mean-field limits of particle systems interacting through hitting times, which is important for systemic risk modeling.
Our main contributions are: (i) A general tightness theorem for the Skorokhod $M_1$-topology which applies to processes that can be decomposed into a continuous and a monotone part.  (ii) A propagation of chaos result for a perturbed version of the particle system for general initial conditions. (iii) The proof of a conjecture of Delarue, Nadtochiy and Shkolnikov, relating the solution concepts of so-called minimal and physical solutions, showing that minimal solutions of the McKean--Vlasov equation are physical whenever the initial condition is integrable.
\end{abstract}


\section{Introduction}
\subsection{The classical supercooled Stefan problem}
Stefan problems are models for the evolution of the interface 
between two phases of a substance undergoing a phase transition. Historically, the study of these problems goes back to the eponymous physicist 
\cite{stefan1891theorie} studying the
growth of ice, and to 
\cite{lame1831memoire},
studying the formation of the earth's crust. The classical one-dimensional supercooled Stefan problem
is a simple model for the freezing of a supercooled liquid on the semi-infinite strip $[0,\infty)$. It can be formulated using the set of equations
\begin{subequations}
\label{eq:SCSP}
\begin{align}
\label{eq:SCSP_heattransport}
\partial_t u &= \frac{1}{2}\partial_{xx}u, & &\L_t < x < \infty,~\quad t > 0, \\
\label{eq:SCSP_isothermal}
u(t,\L_t) &= 0, & &t>0,\\
\label{eq:SCSP_stefancondition}
\frac{\alpha}{2}\partial_{x}u(t,\L_t) &= -\dot{\L}_t, & &t>0,\\  
\label{eq:SCSP_initial}
u(0,x) &= - f(x), & &x>0.
\end{align}
\end{subequations}
Here, we interpret $t,x$ and $u=u(t,x)$ as 
 time, position and temperature, respectively. The freezing point is at $u=0$, and the freezing front (the interface between solid and liquid) at time $t$ is located at position $x = \L_t$. We assume that $\alpha > 0$ is a known constant and that $f$ is a known probability density function. Equation \eqref{eq:SCSP_initial} then implies that the liquid is initially below or at its freezing point - which is precisely the reason why we refer to the problem as \emph{supercooled}. 
Equation \eqref{eq:SCSP_heattransport} describes the heat transport
in the fluid phase. Condition \eqref{eq:SCSP_isothermal} asserts that the phase change is isothermal, and condition \eqref{eq:SCSP_stefancondition} is a so-called \emph{Stefan condition}, 
which balances the discontinuity in heat flux across the freezing
front with the release of latent heat during freezing. 
The set of equations \eqref{eq:SCSP} describes a \emph{one-phase} model, which means that we do not account for heat transport in the solid phase, which amounts to assuming that the temperature in the solid is constant and equal to $u = 0$. 
Solving the supercooled Stefan problem then amounts to finding functions $u$ and $\L$ such that \eqref{eq:SCSP} is satisfied.

It is well-known that for certain initial conditions, problem \eqref{eq:SCSP} exhibits blow-ups in finite time (see \cite{sherman1970general}).
The classical remedy for this issue is a modification of the
boundary condition, see \cite{dewynne1992survey} for a survey and \cite{baker2020zero} for recent developments.
We follow the approach introduced in \cite{delarue2019global}, which
allows us to globally define solutions to \eqref{eq:SCSP} in the 
presence of blow-ups by virtue of a probabilistic reformulation.

\subsection{Probabilistic reformulation}

Following \cite{delarue2019global}, we consider the probabilistic reformulation of the supercooled Stefan problem \eqref{eq:SCSP}, given by the following McKean--Vlasov equation
\begin{equation}\label{eq:mckeanproblem}
\left\{
\begin{aligned}
X_t &= X_{0-} + B_t - \L_t \\ 
\tau &= \inf\{t \geq 0: X_t \leq 0 \} \\
\Lambda_t &=\alpha\P{\tau \leq t},
\end{aligned}\right.
\end{equation}
where the initial condition $X_{0-}$ is supported in $[0,\infty)$, and $B$ is an independent Brownian motion started at 0. (The assumption about the support of $X_{0-}$ is not needed, but natural for the applications we have in mind). A solution to this equation is a 
triple $(X,\tau,\L)$, where $\L: [0, \infty) \to [0,\alpha]$ is a
deterministic, (not necessarily strictly-) increasing c\`adl\`ag function  such that \eqref{eq:mckeanproblem} holds for any Brownian motion $B$, and $X$ and $\tau$ are the resulting process and stopping time, respectively. 

We heuristically motivate the connection between \eqref{eq:SCSP} and \eqref{eq:mckeanproblem} by means of a formal calculation. Suppose that $(X,\tau, \L)$ solves \eqref{eq:mckeanproblem} for some continuously differentiable loss function $\Lambda$, let $\varphi \in C^2$ be a test function with $\varphi(0)=0$ and let 
$X_{0-}$ admit the density $f$. Applying It\^{o}'s formula and taking expectations yields
\begin{align*}
\E{\varphi(X_t) \ind{[\tau > t]}} = \E{\varphi(X_{0-})} + \frac{1}{2}\int_{0}^{t}\E{\varphi''(X_s) \ind{[\tau > s]}}~\mathrm{d}s
-\int_{0}^{t}\E{\varphi'(X_s) \ind{[\tau > s]}}~\mathrm{d}\L_s.
\end{align*}
Denoting the subdensity of $X_{t}\ind{[\tau > t]}$ on $(0,\infty)$ by $p(t,\cdot)$,
integrating by parts yields
$$
\partial_t p = \frac{1}{2} \partial_{xx}p + 
\dot{\L}_t\partial_{x}p, \quad p(0,\cdot) = f, \quad p(\cdot,0) =0.
$$
Taking the derivative with respect to time of the equation ${\L_t = \alpha(1- \int_{0}^{\infty} p(t,x)~\mathrm{d}x})$, we see that
\begin{align*}
\dot{\L}_t = -\alpha\int_{0}^{\infty}\partial_{t}p(t,x)~\mathrm{d}x
= -\alpha\int_{0}^{\infty} \frac{1}{2} \partial_{xx}p(t,x)~\mathrm{d}x -\alpha\int_{0}^{\infty}
\dot{\L}_t\partial_{x}p(t,x)~\mathrm{d}x = \frac{\alpha}{2} \partial_{x} p(t,0).
\end{align*}
Therefore, setting $u(t,x):= -p(t,x-\L_t)$, we (formally) obtain a solution to \eqref{eq:SCSP}. As \cite[Theorem 1.1]{delarue2019global} shows, this argument can be made rigorous for so-called \emph{physical solutions} to the McKean--Vlasov problem \eqref{eq:mckeanproblem} in between jump times.

Apart from the physical significance of this equation, systemic risk and neuro-science applications
have recently sparked considerable interest in the
McKean--Vlasov problem \eqref{eq:mckeanproblem} and variants thereof, see e.g.~\cite {delarue2015particle, nadtochiy2019particle, hambly2019mckean, delarue2019global, ledger2020uniqueness}. 
The existence question could be clarified in \cite{delarue2015particle}, where (for a slightly different equation) it is shown that global solutions to \eqref{eq:mckeanproblem} 
can be obtained as limit points of the following particle system
\begin{equation}\label{eq:particleN}
\left\{\begin{aligned}
X_{t}^{i,N} &= X_{0-}^{i} + B_{t}^{i} - \Lambda_t^N \\
\tau_{i,N} &= \inf\{t \geq 0: X_{t}^{i,N} \leq 0\} \\
\Lambda_{t}^N &= \frac{\alpha}{N}\sum_{i=1}^{N} \ind{\left[\tau_{i,N}\leq t\right]},
\end{aligned}\right.
\end{equation}
with $N \in \N$ and $i=1,\dots,N$.
Here, the initial conditions are iid with the same law as $X_{0-}$ in \eqref{eq:mckeanproblem}, and  $B^i$ are independent Brownian motions independent of $(X_{0-}^{i})_{i=1}^{N}$. A solution to this equation is a 
triple $(X^N,\tau_N,\Lambda^N)$, where $\Lambda^N$ is a $[0,\alpha]$-valued, increasing, c\`adl\`ag process such that \eqref{eq:mckeanproblem} holds for the Brownian motions $(B^i)_{i=1}^{N}$, and $X^N$ and $\tau_N$ are the resulting 
$N$-dimensional processes and stopping times, respectively.

\subsection{Connections to systemic risk}

From the particle system \eqref{eq:particleN} the connection to systemic risk in finance becomes apparent. Indeed, consider $N$ banks and assume that $X^{i,N}$ stands for the evolution of bank $i$'s equity value (or rather the distance-to-default, see \cite{feinstein2021dynamic} and the introduction of \cite{hambly2019mckean} regarding possible economic interpretations). The dynamics described by \eqref{eq:particleN} 
then correspond to a perfectly homogeneous lending network, where every bank lends a quantity of $\alpha/N$ to every other bank in 
the system. Whenever one $X^{i,N}$ hits $0$, then -- due to the exposure to all other banks -- a contagion effect occurs. It results in a loss of $\alpha/N$ for all other banks, which in turn can cause further defaults leading to default cascades. Such default cascades can trigger so-called \emph{systemic events,} whose definition, 
introduced in \cite{nadtochiy2019particle}, is based on the limiting McKean--Vlasov equation \eqref{eq:mckeanproblem}. Indeed, for  a solution $(X,\tau,\L)$ to \eqref{eq:mckeanproblem},  the function $\L$ can be interpreted as the aggregate loss of a typical bank in a large banking system caused by defaults from other banks. Mathematically, a \emph{systemic event} is then defined as a jump-discontinuity of $\Lambda$, meaning that 
a considerable fraction of banks defaults in an instant. For more on the relevance of models such as \eqref{eq:mckeanproblem} to the topic of systemic risk, we refer the reader to \cite{hambly2019spde}. 

A natural question at this point is under which conditions jump-discontinuities of $\L$ occur. 
Not surprisingly, the smoothness of $\L$ depends on the tuple $(X_{0-},\alpha).$ A remarkably 
short argument \cite[Theorem 1.1]{hambly2019mckean} shows that for any solution $(X,\tau,\L)$ to the McKean--Vlasov problem \eqref{eq:mckeanproblem}, the function $\L$ must
be discontinuous whenever $2\E{X_{0-}} < \alpha.$ Conversely, in the weak feedback regime (if the feedback parameter $\alpha$ is ``small'' relative to the the initial condition), global uniqueness and continuity of $\L$ is proved in \cite{ledger2020uniqueness}. A similar result is obtained in \cite{delarue2015global} for deterministic initial conditions. 

\subsection{Solution concepts for the McKean--Vlasov problem} \label{sec:solutionconcepts}
The preceding definition of a systemic event only makes sense when a suitable \emph{propagation of chaos-result} for the sequence of empirical measures  that corresponds to the particle system \eqref{eq:particleN} can be established.
By the results of \cite{delarue2015particle}, this holds true provided that there is uniqueness among the limit points. But this is exactly (one of) the crucial and still unresolved issue(s): it is open whether uniqueness of limit points of the particle system and uniqueness of solutions to \eqref{eq:mckeanproblem} hold within the class of so-called \emph{physical solutions} for general initial conditions. Following \cite{delarue2015particle}, we call a solution
$(X,\tau,\L)$ to \eqref{eq:mckeanproblem} \emph{physical}, if 
\begin{align}\label{eqdef:physicaljump}
\Delta \L_t = \alpha\inf\{x > 0 \colon \P{\tau \geq t, X_{t-} \in [0,\alpha x]} < x \}, \quad t \geq 0.
\end{align}
Proposition 1.2 in \cite{hambly2019mckean} reveals that for each instant in time, \eqref{eqdef:physicaljump} amounts to choosing the smallest possible jump size of $\L$ that allows for a \cadlag solution  of \eqref{eq:mckeanproblem}. In particular, solutions whose loss function $\L$ is continuous are physical.

There is an analogue of the notion of physical solution for the particle system; we provide more details in Section~\ref{sec:physical_solutions}. In the particle system, restricting to physical solutions
is economically and physically meaningful, and allows to preclude economically elusive solutions and to conclude uniqueness of \eqref{eq:particleN} (for every fixed $N$) 
among the set of physical solutions. Moreover, it turns out  that the physical solution coincides with  another meaningful type of solution, namely the \emph{minimal solution} (see Lemma~\ref{lemma:physicalminimal}). 

In the current paper we shall focus on this solution concept and illustrate that this is in several respects the right way to look at the problem.
We call a solution $(\Xmin^N,\taumin_N,\Linfmin^N)$ to \eqref{eq:particleN} \emph{minimal} if, almost surely, for any other solution $(X^N,\tau_N,\L^N)$ to \eqref{eq:particleN} coupled to the same Brownian motions $(B^i)_{i=1}^{N}$, we have 
\begin{equation}\label{eq:minimalNdef}
\Linfmin_t^N \leq \L_t^N, \quad t \geq 0.
\end{equation} 
In an analogous manner, we call a solution $(\Xmin,\taumin,\Linfmin)$ to the McKean--Vlasov problem \eqref{eq:mckeanproblem} \emph{minimal},
if for every solution $(X,\tau, \L)$ to \eqref{eq:mckeanproblem} we have
\begin{equation}\label{eq:minimaldef}
\Linfmin_t \leq \L_t, \quad t \geq 0.
\end{equation}
Note that this condition is deterministic as $\Linfmin$ and $\L$ are necessarily deterministic.

Significant progress regarding the question of uniqueness of physical solutions has been made recently, 
going beyond the weak feedback regime. In \cite{delarue2019global}, uniqueness of the physical solution to the McKean--Vlasov problem \eqref{eq:mckeanproblem} is established for initial conditions with a bounded and non-oscillatory density for any $\alpha > 0$.  However, as mentioned above uniqueness within the class of physical solutions for general $X_{0-}$ and the question whether the (unique) minimal solution coincides with (one of) the physical solution(s), are open problems.

\subsection{Main results}

\begin{figure}
\centering
\includegraphics[width=0.45\textwidth,height=0.25\textheight]{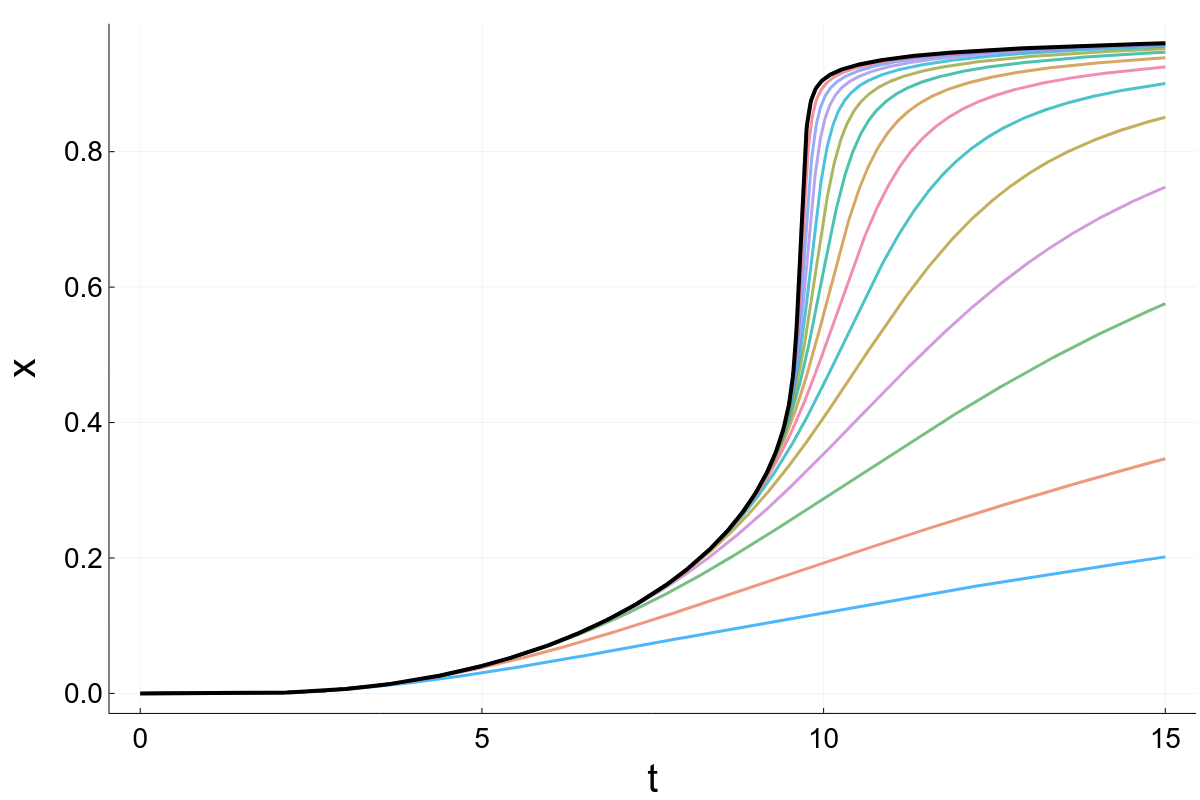}
\includegraphics[width=0.5\textwidth,height=0.25\textheight]{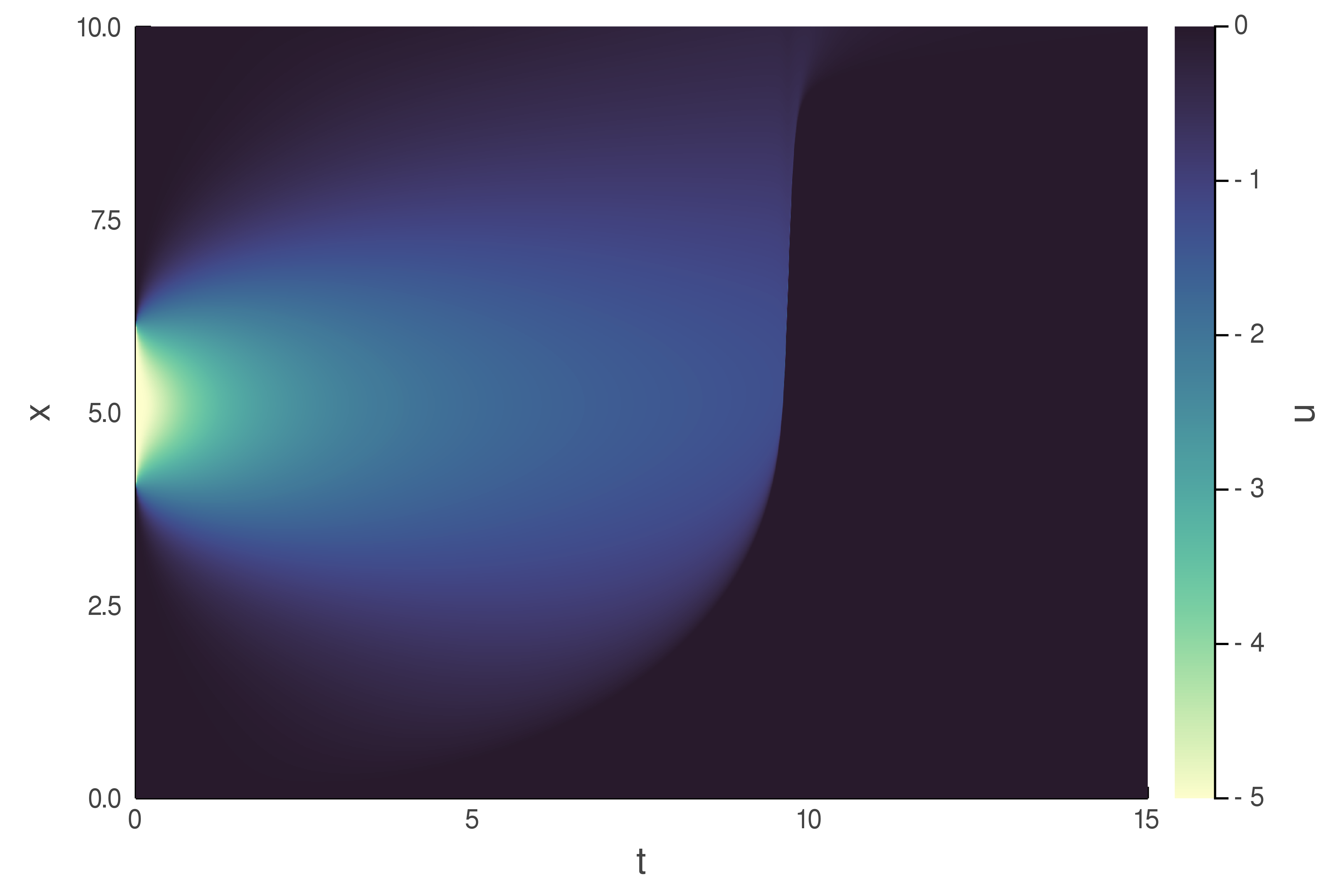}
\caption{The left-hand side shows the iterates $\Gamma^{(k)}[0]$. The right-hand side shows the solution $u(t,x)$ to the supercooled Stefan problem. The parameters were $\alpha=10$, and $X_{0-}$ was chosen to be uniformly distributed on (4,6).}
\label{fig:iterationandheatmap}
\end{figure}

One of the main contributions of the current paper is to answer the latter question raised by \cite{delarue2019global}
with ``yes'':\footnote{Note that the conjecture was stated on p.55 in v1 of the arXiv preprint of \cite{delarue2019global} and has been removed from the updated version.} \emph{the minimal solution is a physical solution} (see Theorem~\ref{thm:minimalisphysical}). The key for proving this is a propagation of chaos-result for a perturbed particle system. Indeed, consider  \eqref{eq:particleN} with an initial condition that is slightly larger, namely $X_{0-}^i+N^{-\gamma}$, for $\gamma \in (0, \frac{1}{2})$. 
Then we can show that the empirical distribution of the minimal (and thus physical) solution to the perturbed particle system converges in probability to the law of the minimal solution to \eqref{eq:mckeanproblem}. This is what we call \emph{perturbed propagation of minimality} (see Theorem~\ref{thm:pertpropagation}). Combining this with the (known) fact that physical solutions of the (perturbed) particle system converge to physical solutions of the McKean--Vlasov equation, we can thus conclude.



Part of the motivation for studying the concept of minimal solutions was the question whether one could more easily prove propagation of chaos using this concept, as the methods employed in \cite{delarue2019global} are not easily applicable to more general driving processes and initial conditions. While we believe that the methods presented here can be generalized in many ways, we only prove propagation of chaos with a perturbed initial condition (see Theorem \ref{thm:pertpropagation}), and still rely on the uniqueness result of \cite{delarue2019global} to obtain ``propagation of minimality'' for the unperturbed system. However, if the minimal solution of the McKean--Vlasov equation were stable under additive perturbations of the initial condition, our methods would yield full propagation of chaos. We formulate this as a conjecture at the end of the paper.

Our results also allow to exploit iteration \eqref{eq:iteration} for numerical purposes (for more on numerical methods for this problem, see \cite{kaushansky2019simulation}, \cite{kaushansky2020convergence} and \cite{lipton2019semi}). 
Applying a time-discretziation and iteratively solving first-passage time problems as in \cite[p.230]{peskir2006optimal}, one can calculate the minimal (and thus physical) solution numerically. Figure \ref{fig:iterationandheatmap} illustrates an implementation of this method. Rigorously establishing convergence of this scheme is however nontrivial and beyond the scope of this article.

In the following let us summarize the main contributions of the current article. 
\begin{itemize}
\item A general tightness result in the $M_1$ topology for stochastic processes that can be decomposed into a continuous  and  an increasing \cadlag process (Theorem~\ref{thm:Dinftytightness}).
\item (Perturbed) Propagation of minimality (Theorem~\ref{thm:pertpropagation} and Theorem~\ref{cor:propmin}) via the trilogy of arguments consisting of
tightness with respect to Skorokhod's $M_1$ topology (Corollary~\ref{cor:tightness}), convergence of solutions (Proposition~\ref{thm:finitedimconvergence}) and identification of the limit (Section~\ref{sec:identification}).
\item The physical nature of the minimal solution and hence global existence of physical solutions whenever the initial condition is integrable (Theorem~\ref{thm:minimalisphysical}).
\end{itemize}

The remainder of the article is structured as follows.
 In Section \ref{sec:minimalmckean}, we construct the minimal solution through a fixed-point iteration. In Section \ref{sec:solutionsparticle}, we discuss the notions of physical and minimal solutions for the particle system and show that they are equivalent. Section \ref{sec:tightness} is dedicated to proving tightness of  the 
empirical measures associated to the particle system, in Section \ref{sec:convergence} we show that limit points of such empirical measures correspond to solutions of the McKean--Vlasov problem in a certain sense.
Finally, in Section \ref{sec:identification}, we prove propagation of minimality under various perturbations and deduce that the minimal solution of the McKean--Vlasov problem is physical whenever the initial condition is integrable.

\section{The minimal solution of the McKean--Vlasov problem}
\label{sec:minimalmckean}

In this section, we follow the same strategy employed in \cite{hambly2019mckean}, \cite{nadtochiy2020mean} and \cite{delarue2015global}, which is to 
decouple system \eqref{eq:mckeanproblem} by rewriting it as a fixed-point problem for a certain operator. In \cite{delarue2015global}, global well-posedness
results are shown in the weak feedback regime (i.e., for small $\alpha$) for deterministic initial conditions. In the work \cite{nadtochiy2020mean}  more general networks are considered and a version of the Schauder-Tychonoff theorem is proved for the Skorokhod $M_1$-topology, however these results do not prove the existence of a minimal or physical solution. In \cite{hambly2019mckean}, existence (and uniqueness) of solutions 
is shown up to the time of the first discontinuity under some regularity assumptions on the initial condition. The results presented in the current article
are global in time and show the global existence of minimal solutions without any restrictions on the initial condition $X_{0-}$ and the feedback 
parameter $\alpha$. Note that global existence of minimal solutions was first shown in a preprint version of \cite{delarue2019global}, 
however this result was removed from subsequent versions of said paper and we follow a somewhat different proof strategy here.

Define the operator $\Gamma$ for a \cadlag function $\ell$ as
\begin{equation}\label{eqdef:Gamma}
\left\{
\begin{aligned}
X_t^{\ell} &= X_{0-} + B_t - \alpha \ell_t \\ 
\tau^{\ell} &= \inf\{t \geq 0: X_t^{\ell} \leq 0 \} \\
\Gamma[\ell]_t &= \P{\tau^{\ell} \leq t}.
\end{aligned}\right.
\end{equation}
Note here that $(X^\ell,\tau^\ell,\alpha\ell)$ solves \eqref{eq:mckeanproblem} if and only if $\ell$ is a fixed-point of  $\Gamma$.

 It is straightforward to see that $\Gamma$ is monotone in the sense that
\begin{equation}\label{eq:GammaMonotone}
\ell^1_t \leq \ell^2_t, \quad t \geq 0 \quad \implies \quad \Gamma[\ell^1]_t \leq \Gamma[\ell^2]_t, \quad t \geq 0.
\end{equation} 
We are now interested in finding a space on which the operator $\Gamma$ stabilizes. Let $\overline{\R}$ denote the extended real line (i.e., the two-point compactification of $\R$). As a consequence of $\overline{\R}$ being a compact metric space and $[0,\infty]$ being a closed subset thereof, the space of probability measures on $[0,\infty]$, denoted as $\mathcal{P}([0,\infty])$, endowed with the topology of weak convergence of probability measures is a compact Polish space (for more details see e.g.~\cite{klenke2013probability})\footnote{Note that here we use the language of probabilists, this mode of convergence corresponds to the weak*-convergence of measures in the language of functional analysis.}. Set
\begin{equation}\label{eq:DistFunctiondef}
\DistFunctions := \{ \ell \colon \overline{\R} \rightarrow [0,1]~|~  \ell \text{ \cadlag and increasing, }~\ell_{0-} = 0,~\ell_{\infty} = 1\},
\end{equation}
then we may identify the elements of $\DistFunctions$ with distribution functions of measures in $\mathcal{P}([0,\infty])$ via the map $\ell \mapsto \mu_{\ell}$, where we set 
$\mu_{\ell}([0,t]) := \ell_t$ for $t\geq 0.$
Convergence in $\DistFunctions$ is then equivalent to weak convergence of probability measures in $\mathcal{P}([0,\infty])$. In particular, we have that $\ell^n \to \ell$ in $\DistFunctions$ if and only if $\mu_{\ell^n} \to \mu_{\ell}$ weakly in $\mathcal{P}([0,\infty])$, if and only if $\ell_{t}^n \to \ell_t$ for all $t \in [0,\infty]$ that are continuity points of $\ell$. Equipped with this topology, $\DistFunctions$ is a compact Polish space. We now show that $\Gamma$ is a continuous operator on $\DistFunctions$.
\begin{proposition}\label{thm:Gammacont}
The operator $\Gamma: \DistFunctions \to \DistFunctions$ is continuous.
\end{proposition}
\begin{proof} 
Let $\ell^n \to \ell$ in $\DistFunctions$. We show that $\Gamma[\ell^n]_t \to \Gamma[\ell]_t$ for all $t \in [0,\infty]$ that are continuity points of $\Gamma[\ell]$.
%
%
We begin by proving
\begin{align}\label{eq:limsupgamma}
\limsup_{n\to\infty} \Gamma[\ell^n]_t \leq \Gamma[\ell]_t, \quad t \geq 0. 
\end{align}
To prove \eqref{eq:limsupgamma}, note that since $\Gamma[\ell]_t = \P{\exists s \in [0,t] : X_{0-} + B_s \leq \alpha \ell_s}$, by the reverse Fatou lemma it is enough to show that
\begin{align}\label{eq:indinequality}
\limsup_{n\to\infty} \ind{\{\exists s \in [0,t] : X_{0-} + B_s \leq \alpha \ell_s^n\}} \leq \ind{\{\exists s \in [0,t] : X_{0-} + B_s \leq \alpha \ell_s\}}
\end{align}
holds almost surely. If the left-hand side of \eqref{eq:indinequality} is equal to zero there is nothing to prove, so suppose that $\omega \in \Omega$ is such that the left-hand side of \eqref{eq:indinequality} is equal to one. Then, after passing to subsequences if necessary, we can find a sequence $s_n \to s$ with $s \in [0,t]$ such that $X_{0-}(\omega) + B_{s_n}(\omega) \leq \alpha \ell_{s_n}^n$. Let $s' > s$ be a continuity point of $\ell$, then since the $\ell^n$ are increasing we find $X_{0-}(\omega)+B_{s}(\omega) \leq \alpha\limsup_{n\to\infty} \ell_{s^n}^{n} \leq \limsup_{n\to\infty} \alpha\ell_{s'}^{n} = \alpha\ell_{s'}$. Letting $s'$ tend to $s$ we obtain $X_{0-}(\omega) + B_{s}(\omega) \leq \alpha \ell_{s}$, which yields \eqref{eq:indinequality} and hence \eqref{eq:limsupgamma}.

Let $t = 0$. If zero is not a point of continuity for $\Gamma[\ell]$, there is nothing to prove. If zero is a point of continuity, then $\Gamma[\ell]_0 = 0$. Inequality \eqref{eq:limsupgamma} shows that $\Gamma[\ell^n]_0$ goes to $0$ as $n$ goes to infinity. 

Now let $t>0$ be a point of continuity for $\Gamma[\ell]$. We claim that
\begin{equation}\label{eqn8}
\lim_{n\to\infty} (\Gamma[\ell]_t - \Gamma[\ell^n]_t)^+ = 0.
\end{equation}
We may write
\begin{align*}
(\Gamma[\ell]_t - \Gamma[\ell^n]_t)^+ 
\leq \P{\tau^{\ell^n} > t, \tau^\ell \leq t } 
= \int_{[0,t]} \P{\tau^{\ell^n} > t ~|~ \tau^\ell = s} d\Gamma[\ell]_s.
\end{align*}
We now split up the integrand in its continuous and jump part, writing $\Gamma[\ell]_s^c$ for the continuous part. Then, following the proof of Proposition 3.1 in \cite{hambly2019mckean}, we find 
\begin{align}\label{eq:integralcontinuous}
&\int_{0}^{t} \P{\tau^{\ell^n} > t ~|~ \tau^\ell = s} d\Gamma[\ell]_s^c \leq  
\int_{0}^{t} \left(2\Phi\left( \alpha\frac{\ell_s - \ell_s^n}{\sqrt{t-s}} \right) - 1 \right) d\Gamma[\ell]_s^c,
\end{align}
where $\Phi$ denotes the cumulative distribution function of a 
standard normal random variable. Because the set of discontinuity times of $\ell$ is at most countable, the integrand in \eqref{eq:integralcontinuous} converges $\Gamma[\ell]^c$-almost everywhere to $0$. Consequently, the integral in \eqref{eq:integralcontinuous} vanishes as $n\to\infty$ by the dominated convergence theorem. 
 
Regarding now the integral with respect to the jump part of $\Gamma[\ell]$, we obtain
\begin{align}\label{eq:sumdiscontinuous}
\int_{0}^{t} \P{\tau^{\ell^n} > t ~|~ \tau^\ell = s} d(\Gamma[\ell] -\Gamma[\ell]^c)_s \nonumber 
=&\sum_{s \leq t} \P{\tau^{\ell^n} > t ~|~ \tau^\ell = s} \Delta \Gamma[\ell]_s \nonumber \\
=& \sum_{s < t} \P{\tau^{\ell^n} > t ~|~ \tau^\ell = s}\P{\tau^\ell = s} \nonumber \\
=&\sum_{s < t} \P{\tau^{\ell^n} > t, \tau^\ell = s},
\end{align}
where we may take the sum over $s < t$ because $t$ was assumed to be a continuity point of $\Gamma[\ell]$.
Again, by the reverse Fatou lemma and the Portmanteau theorem, we find
\begin{align*}
\limsup_{n\to\infty} \P{\tau^{\ell^n} > t , \tau^\ell = s}  &\leq 
\P{\forall \varepsilon \in (0,t-s): X_{0-}+B_{s+\varepsilon} \geq \alpha \ell_{s+\varepsilon-} , \tau^\ell = s} \\
&\leq \P{X_{0-}+B_{s} \geq \alpha \ell_{s} , \tau^\ell = s} \\
&\leq  \P{X_{0-} + B_s = \alpha \ell_{s}} \\
&= 0.
\end{align*}
By the dominated convergence theorem we find that the sum in \eqref{eq:sumdiscontinuous} converges to zero as $n$ goes to infinity. 
Together with \eqref{eq:integralcontinuous} this yields claim \eqref{eqn8}, which, combined with \eqref{eq:limsupgamma}, proves
$
\lim_{n\to\infty} \Gamma[\ell^n]_t = \Gamma[\ell]_t.
$
\end{proof}


Using the monotonicity of the operator $\Gamma$, one can iteratively construct the minimal solution to the McKean--Vlasov equation, an idea due to \citep{delarue2019global}.
 Figure~\ref{fig:iterationandheatmap} illustrates this iteration.

 \begin{remark}
Note that we did not define $\Gamma[\ell]_{\infty}$, but by definition, if $\Gamma[\ell] \in \DistFunctions$ should hold, we need to set $\Gamma[\ell]_{\infty} := 1$. We will continue to only define elements of $M$ on $[0,\infty)$, as the value at infinity
necessarily needs to be equal to one. 
In particular, we will often consider the function $t \mapsto 0$ as an element of $M$, and really mean $t \mapsto \ind{\{\infty\}}(t)$ on $[0,\infty]$.
\end{remark}

\begin{proposition}\label{lem:miniterationinfty} For any initial condition $X_{0-}$ and $\alpha > 0$, there is a minimal solution to \eqref{eq:mckeanproblem}, which we denote by $(\Xmin,\Linfmin)$. It holds that
\begin{align}\label{eq:iteration}
\alpha \lim_{k\to\infty} \Gamma^{(k)}[0] = \Linfmin, 
\end{align}
in $\DistFunctions$, where $\Gamma^{(k)}$ denotes the $k$-th iterate of the operator $\Gamma$ as defined in $\eqref{eqdef:Gamma}$.
\end{proposition}
\begin{proof}
By definition, we have that $0 \leq \Gamma[0]$. Using the monotonicity of $\Gamma$, this implies that 
\begin{align*}
\Gamma[0] \leq  \Gamma[\Gamma[0]] = \Gamma^{(2)}[0]
\end{align*}
and a straightforward induction shows that
$
\Gamma^{(k)}[0] \leq \Gamma^{(k+1)}[0]$, for each $k \in \N$.
The sequence $(\Gamma^{(k)}[0]_t)_{k\in \N}$ is therefore increasing and bounded by $1$ for every $t \geq 0$, which implies
that we may define $\tilde{\L}$ to be the pointwise limit
\begin{equation*}
\tilde{\L}_t := \alpha\lim_{k\to\infty} \Gamma^{(k)}[0]_t, \quad t \geq 0.
\end{equation*}
Clearly, $\tilde{\L}$ is increasing with $\tilde{\L}_{0-} = 0$ and $\tilde{\L}_{\infty} = \alpha$, so its \cadlag modification $\Linfmin_t := \tilde{\L}_{t+}$ 
lies in $\alpha\DistFunctions$. By construction, we have that 
$\lim_{k\to\infty}\Gamma^{(k)}[0] = \frac{1}{\alpha}\Linfmin$
in $\DistFunctions$. By continuity of $\Gamma$ on $\DistFunctions$ (Proposition~\ref{thm:Gammacont}), we obtain
\begin{equation*}
\alpha\Gamma\left[\frac{1}{\alpha}\Linfmin\right] = \alpha\Gamma[\lim_{k\to\infty}\Gamma^{(k)}[0]] = \alpha\lim_{k\to\infty}\Gamma^{(k+1)}[0] = \Linfmin, 
\end{equation*}
so $\Linfmin$ solves \eqref{eq:mckeanproblem}. Now suppose that $\L$
is another solution to \eqref{eq:mckeanproblem}. By definition, it holds that $\L \geq 0$, and using the monotonicity of $\Gamma$ this leads to  
\begin{align*}
\alpha\Gamma[0] \leq \alpha\Gamma\left[\frac{1}{\alpha}\L\right] = \L.
\end{align*}
A straightforward induction shows that
$
\alpha\Gamma^{(k)}[0] \leq \L$, for each  $k \in \N$.
If $t$ is a continuity point of $\Linfmin$, this implies
\begin{equation}
\Linfmin_t = \alpha\lim_{k\to\infty}\Gamma^{(k)}[0]_t \leq \L_t,
\end{equation}
which by right-continuity implies $\Linfmin \leq \L$. As $\L$ was an arbitrary solution to \eqref{eq:mckeanproblem}, this proves that $\Linfmin$ is in fact the minimal solution.
\end{proof}
The next example illustrates how the solution to the McKean--Vlasov problem \eqref{eq:mckeanproblem} can fail to be unique.
\begin{example}[Several solutions to the McKean--Vlasov problem]
Let $X_{0-} \equiv 1$ and set $\alpha = 1$. Then, $\overline{\L} \equiv 1$ solves the McKean--Vlasov problem \eqref{eq:mckeanproblem}.
Later on, we will prove (see Theorem~\ref{thm:minimalisphysical}) that the minimal solution is physical if the initial condition is integrable, which implies by definition that
\begin{equation*}
\Linfmin_0 = \Linfmin_0 - \Linfmin_{0-} = \alpha\inf\{x \geq 0: \P{X_{0-} \in [0,x]} < x\} = 0.
\end{equation*}
Therefore, clearly we have $\Linfmin \neq \overline{\L}$, and hence nonuniqueness of the McKean--Vlasov problem \eqref{eq:mckeanproblem}.
\end{example}

\section{Solutions of the particle system}
\label{sec:solutionsparticle}
\subsection{Minimal solutions}
Although it may seem intuitively obvious, we have not shown yet that there is a minimal solution to the particle system \eqref{eq:particleN}.
A natural question at this point is whether the same construction outlined in Lemma~\ref{lem:miniterationinfty} works, and a moment's reflection
shows that it does. Define the operator $\Gamma_N$ via

\begin{equation}\label{eqdef:GammaN}
\left\{\begin{aligned}
X_{t}^{i,N}[\mathsf{L}] &= X_{0-}^{i} + B_{t}^{i} - \alpha \mathsf{L}_t \\
\tau_{i,N}[\mathsf{L}] &= \inf\{t \geq 0: X_{t}^{i,N}[\mathsf{L}] \leq 0\} \\
\Gamma_N[\mathsf{L}]_t &= \frac{1}{N}\sum_{i=1}^{N} \ind{\left[\tau_{i,N}[\mathsf{L}]\leq t\right]}, 
\end{aligned}\right.
\end{equation}
where $\mathsf{L}$ is some \cadlag process. In analogy to \eqref{eq:GammaMonotone}, $\Gamma_N$ is monotone in the sense that 
$$
\mathsf{L}^1_t \leq \mathsf{L}^2_t, \quad t \geq 0 \quad \implies \quad \Gamma[\mathsf{L}^1]_t \leq \Gamma[\mathsf{L}^2]_t, \quad t \geq 0.
$$
 Making use of this montonicity, we readily see by straightforward induction arguments that 
\begin{equation}\label{eq:GammaNiterationprop}
\alpha\Gamma_N^{(k)}[0] \leq \L^N, \quad \Gamma_N^{(k)}[0] \leq \Gamma_N^{(k+1)}[0], \quad k \in \N,
\end{equation}
holds almost surely, where $\L^N$ is any solution to the particle system
and $\Gamma_N^{(k)}$ denotes the $k$-th iterate of $\Gamma_N$. We did not show any suitable continuity properties of $\Gamma_N$ to conclude in the 
same way as in Lemma~\ref{lem:miniterationinfty}, but as we prove in the next lemma, the iteration $(\Gamma_N^{(k)}[0])_{k\in\N}$ is constant
after at most $N$ steps.
\begin{lemma}\label{lem:miniteration}
For $N \in \N$, let $\Gamma_N$ be defined as in \eqref{eqdef:GammaN}.
Then  $\Linfmin^N:=\alpha\Gamma^{(N)}_N [0]$ is the minimal solution to the particle system 
\eqref{eq:particleN} and the error bound
\begin{align}\label{eq:nerrorbound}
\|\alpha\Gamma^{(k)}_N[0] - \Linfmin^N\|_{\infty} \leq  \alpha\frac{(N-k)^+}{N}
\end{align}
holds almost surely. 
\end{lemma}
\begin{proof}
For $k \in \N$, define the stopping times 
\begin{align*}
\sigma_k := \inf\big\{t \geq 0 \colon \Gamma_N^{(k)}[0]_t \geq k/N\big\}.
\end{align*}
First, we show by induction on $k$ that 
\begin{align}\label{eq:miniterationinductivebasis}
\Gamma_N^{(k-1)}[0]_t = \Gamma_N^{(k)}[0]_t, \quad t < \sigma_k,
\end{align}
for each $k\in\N$.
For the base case, let $k=1$ (we define $\Gamma_N^{(0)}$ to be the identity operator). Observe that the first time that $\Gamma_N[0]$ jumps coincides with the first time any of the Brownian motions $(X_{0-}^{i} +B^i)_{i\in\N}$ hit $0$, and $\Gamma_N[0]$ is equal to $0$ before that time. This means that we have
$
0 = \Gamma_N^{(0)}[0]_{t}=\Gamma_N^{(1)}[0]_{t}
$
for each $t < \sigma_1$.
For the inductive step, assume the claim holds for all natural numbers up to $k$. Applying $\Gamma_N$ to both sides of \eqref{eq:miniterationinductivebasis}, we obtain 
\begin{align}
\Gamma_N^{(k)}[0]_t = \Gamma_N^{(k+1)}[0]_t, \quad t < \sigma_{k}. 
\end{align}
We distinguish two cases: In the first case, suppose that $\Gamma_N^{(k)}[0]_{\sigma_{k}} > \frac{k}{N}.$ Due to \eqref{eq:GammaNiterationprop}, we then must have $\Gamma_{N}^{(k+1)}[0]_{\sigma_{k}} \geq \frac{k+1}{N}$ and hence $\sigma_{k} = \sigma_{k+1}$, completing the inductive step. In the second case, we have $\Gamma_N^{(k)}[0]_{\sigma_{k}} = \frac{k}{N}$, and using \eqref{eq:GammaNiterationprop} we find for $t \in (\sigma_k, \sigma_{k+1})$
\begin{align*}
\frac{k}{N} = \Gamma_N^{(k)}[0]_{\sigma_{k}} \leq \Gamma_N^{(k)}[0]_t \leq \Gamma_N^{(k+1)}[0]_{t} < \frac{k+1}{N}
\end{align*}
which shows that $\Gamma_N^{(k)}[0]$ and $\Gamma_N^{(k+1)}[0]$ agree
on all of $[0,\sigma_{k+1})$, completing the inductive step.

Having established \eqref{eq:miniterationinductivebasis}, we show 
that $\Linfmin^N = \alpha\Gamma_N^{(N)}[0]$ solves the particle system \eqref{eq:particleN}. For $k \in \N$, repeatedly applying $\Gamma_N$ to \eqref{eq:miniterationinductivebasis}, we find
\begin{align}\label{eq:miniterationGammakandLmin}
\Gamma_N^{(k)}[0]_t = \Gamma_N^{(k+1)}[0]_t = \hdots = \frac{1}{\alpha}\Linfmin^N_t = \Gamma_N\left[\frac{1}{\alpha}\Linfmin^N\right]_t, \quad t < \sigma_{k}. 
\end{align} 
Choosing $k = N$, we obtain $\Linfmin^N = \alpha\Gamma_N\left[\frac{1}{\alpha}\Linfmin^N\right]$ on $[0,\sigma_N)$. Recall that by  \eqref{eq:GammaNiterationprop} it 
holds that $\Linfmin^N \leq \alpha\Gamma_N\left[\frac{1}{\alpha}\Linfmin^N\right]$, so we see that $\alpha = \Linfmin^N_{\sigma_N} \leq \alpha\Gamma_N\left[\frac{1}{\alpha}\Linfmin^N\right]_{\sigma_N} \leq \alpha$ and therefore $\Linfmin^N = \alpha\Gamma_N[\frac{1}{\alpha}\Linfmin^N]$.
By \eqref{eq:GammaNiterationprop}, it follows that $\Linfmin^N$ is indeed the minimal solution.\\   

The error bound \eqref{eq:nerrorbound} is now straightforward: By \eqref{eq:miniterationGammakandLmin}, $\alpha\Gamma^{(k)}_N[0]$ agrees with $\Linfmin^N$ for $t < \sigma_k$ and is increasing in $t$, and therefore
\begin{align*}
\sup_{t\geq 0} |\Linfmin_t^N - \alpha\Gamma_N^{(k)}[0]_t| = \sup_{t\geq\sigma_k} |\Linfmin_t^N - \alpha\Gamma_N^{(k)}[0]_t| \leq \alpha - \alpha\Gamma_N^{(k)}[0]_{\sigma_k}\leq \alpha \left(1 - \frac{k}{N}\right),
\end{align*} 
for each $k \leq N$.
\end{proof}

As mentioned in the introduction, in the weak feedback regime the solution to the McKean--Vlasov problem \eqref{eq:mckeanproblem} is unique. As the examples in Section 3.1.1 in \cite{delarue2015particle} show, no such condition can guarantee uniqueness of solutions for the particle system.



\subsection{Physical solutions}\label{sec:physical_solutions}

\begin{definition}\label{def:lawstoppedparticle}
If $X^N$ is a solution process to the particle system \eqref{eq:particleN}, define the random subprobability measure 
\begin{align*}
\lawstopped_{t-}^{N} := \frac{1}{N} \sum_{i=1}^{N} \delta_{X_{t-}^{i,N}} \ind{[\tau^{i,N}\geq t]}
\end{align*}
for $t \geq 0$. We call a solution $(\Xphys^N,\Lphys^N)$ of \eqref{eq:particleN} \emph{physical}, if we have
\begin{align}\label{eq:physicalparticlejump}
\Delta \Lphys_t^N = \alpha \inf\left\{\frac{k}{N} \geq 0:~k \in \N,~ \lawstopped_{t-}^{N}\left(\left[0,\alpha \frac{k}{N}\right]\right) \leq \frac{k}{N}\right\}, \quad t \geq 0.
\end{align}
\end{definition}

 Note that $N \lawstopped_{t-}^{N}([0,\alpha \frac{k}{N}])$ is the 
number of particles surviving up to time $t$ which would not survive a kick of $\alpha \frac{k}{N}$. Therefore we see that \eqref{eq:physicalparticlejump} corresponds to choosing the smallest possible jump size at any $t \geq 0$. We recognize formula \eqref{eq:physicalparticlejump} as the discrete analogue of the physical jump condition.

\begin{lemma}\label{lemma:physicalminimal}
The physical solution to \eqref{eq:particleN} is equal to the minimal solution to \eqref{eq:particleN}.
\end{lemma}
\begin{proof}
Note that the physical solution $(\Xphys^N, \Lphys^N)$ to \eqref{eq:particleN} is pathwise unique, as it is unique between jump times  and \eqref{eq:physicalparticlejump} uniquely specifies the size of the jump at any given time. Since the minimal solution $(\Xmin^N, \Linfmin^N)$ is unique by definition, it is  sufficient to show that the physical solution is minimal. 

Let $\sigma$ be the first time the physical solution $(\Xphys^N, \Lphys^N)$ jumps, then for all $t < \sigma$ we have $\Lphys_{t}^N = \Linfmin_{t}^N = 0$, because if the first jump of the minimal solution would happen before $\sigma$ this would contradict the minimality of $(\Xmin^N,\Linfmin^N)$. As physical solutions have minimal jumps, we find $\Lphys_{\sigma}^N \leq \Linfmin_{\sigma}^N$, which due to the minimality of $\Linfmin^N$ implies $\Lphys_{\sigma}^N = \Linfmin_{\sigma}^N$. Repeating this argument for each jump of the physical solution proves the claim. 
\end{proof}

\section{Tightness}
\label{sec:tightness}
As in the previous works on problem at hand, we will deal with the Skorokhod $M_1$-topology in this paper. We explain this choice and collect some fundamental results regarding the $M_1$-topology in  Section~\ref{appA} in the appendix.

\begin{definition}\label{def:levymetric}
Let $w^n, w \in C([0,\infty))$. We say that $w^n$ converges to $w$ with respect to the topology of compact convergence if $w^n \to w$ in $C([0,T])$ for every $T>0$. We also define $\Levymetric$ to be the L\'evy-metric on $\DistFunctions$, that is if $\ell^1,\ell^2 \in \DistFunctions$ we have 
\begin{align*}
\Levymetric(\ell^1,\ell^2) = \inf\{\varepsilon > 0 ~\colon \ell^1_{t+\varepsilon} + \varepsilon \geq \ell^2_t \geq \ell^1_{t-\varepsilon} - \varepsilon,~ \text{for all } t \geq 0\}.
\end{align*}
\end{definition}

It is well-known that the L\'evy-metric metrizes weak convergence. In order to avoid having to work directly with the unwieldy $M_1$-metric, 
the following theorem comes in handy. Recall that $D([T_0,\infty))$ denotes the space of \cadlag paths from $[T_0,\infty)$ to $\R$ furnished with the $M_1$-topology (see Definition~\ref{def:M1infty} for more detail).

\begin{theorem}\label{thm:ExtendedEmbedding}
Define the space $\ExtendedE$ as
\begin{align*}
\ExtendedE := C([0,\infty)) \times \DistFunctions,
\end{align*}
where $\DistFunctions$ is defined as in \eqref{eq:DistFunctiondef}. Endowed with the product topology induced by compact convergence on $C([0,\infty))$ and the L\'evy-metric on $\DistFunctions$, the space $\ExtendedE$ is Polish. For $w\in C([0,\infty))$ and $\ell\in M$, define
\begin{align*}
\hat w_t := \begin{cases} 
w_0 \quad & t \in [-1,0) \\
w_t \quad & t \in [0,\infty)
\end{cases}
\quad
\check \ell_t= \begin{cases} 
0 \quad & t \in [-1,0) \\
\ell_t \quad & t \in [0,\infty).
\end{cases}
\end{align*}
Then, for any $\alpha \in \R$, the embedding $\iota_{\alpha}\colon \ExtendedE \rightarrow D([-1,\infty))$ defined via
$$
\iota_{\alpha}(w,\ell) = \hat w -\alpha\check \ell 
$$
is continuous.
\end{theorem}
\begin{proof}
We show the continuity of $\iota_{\alpha}$. Let $(w^n,\ell^n) \to (w,\ell)$ in
$\ExtendedE$. Let $T>0$ be a continuity point of $\ell$.
By assumption, we have that $w^n \to w$ in $C([0,T])$, which implies that $\widehat{w^n} \to \hat w$ in $C([-1,T])$ and thus in $D([-1,T])$. Again, by assumption, $\ell_t^n$ converges to $\ell_t$ for all $t > 0$ that are continuity points of $\ell$, from which we deduce (using Lemma~\ref{thm:M1monconvergence}) that $-\alpha\widecheck{\ell^n} \to -\alpha\check \ell$ in $D([-1,T])$. By Lemma~\ref{lemma:M1addition}, it follows that
\begin{align*}
\lim_{n\to\infty} \iota_{\alpha}(w^n,\ell^n) = \lim_{n\to\infty} (\widehat{w^n} -\alpha\widecheck{\ell^n}) = \hat w -\alpha\check \ell = \iota_{\alpha}(w,\ell)
\end{align*}
in $D([-1,T])$. The conclusion now follows from Lemma~\ref{lemma:Dinftyconvergence}.
\end{proof}

 It is necessary to extend the domain artificially to the left to obtain
 a continuous embedding of $\ExtendedE$ into the \cadlag functions equipped
 with the $M_1$-topology as in Theorem~\ref{thm:ExtendedEmbedding}.
 The reason lies in the requirements of Lemma~\ref{thm:M1monconvergence}, more precisely the requirement of pointwise convergence in the left interval endpoint. The extension procedure might now look like a cheap trick to circumvent having to show pointwise convergence in $0$ (which it is), but indeed, one easily constructs examples of sequences $(\ell^n)_{n\in \N}$ that converge in $\DistFunctions$, but do  not converge pointwise in $0$ against their \cadlag limit, such as
 $
 \ell^n_t := \ind{[1/n\leq t]}.
 $

\begin{theorem}\label{thm:Dinftytightness}
Let $(X^N)_{N\in \N}$ be a sequence of stochastic processes with paths in $D([0,\infty))$. Suppose that $X^N$ almost surely admits a decomposition
\begin{align}\label{eq:Xdecomposition}
X^N = Z^N - \alpha_N \mathsf{L}^N
\end{align} 
where $\alpha_N$ is a (possibly random) real number, $Z^N$ is continuous, and $\mathsf{L}^N \in M$. Suppose that $(Z^N)_{N\in \N}$ is tight on $C([0,T])$ for each $T>0$ and that $(\alpha_N)_{N\in \N}$ is tight on $\R$.
Set
\begin{align}\label{eq:processformhat}
\hat X_t^{N} := \begin{cases}
Z_{0}^N, \quad & t \in [-1,0), \\
X_{t}^{N}, \quad & t \in [0,\infty).
\end{cases}
\end{align}
Then, the random variables $((Z^N,\mathsf{L}^N))_{N\in\N}$ are tight on $\ExtendedE$ and random variables $(\hat X^{N})_{N\in\N}$ are tight on $D([-1,\infty))$.
\end{theorem}
\begin{proof}
Let $\varepsilon > 0$ and let $r > 0$ be such that 
$
\P{|\alpha_N| > r} < \varepsilon/2
$
or each $N\in\N$.
 Choose a sequence of positive numbers $(T_k)_{k\in \N}$ such that $T_k \nearrow \infty$. Because  $(Z^{N})_{N\in \N}$ is tight on $C([0,T_k])$, we may pick $K_{\varepsilon}^{k} \subseteq C([0,T_k])$ compact, such that
\begin{align*}
\P{Z^{N} \notin K_{\varepsilon}^{k}} < \varepsilon 2^{-(k+1)}, \quad N\in \N.
\end{align*}
Set
$
K_{\varepsilon} := \{ w\in C([0,\infty)): w \in K_{\varepsilon}^{k},~ k \in \N\}.
$
Then $K_\varepsilon$ is compact in the topology of compact convergence (which follows e.g. from Tychonoff's theorem). We obtain that
\begin{align*}
\P{(Z^N,\mathsf{L}^N) \notin K_{\varepsilon} \times \DistFunctions} = \P{Z^N \notin K_{\varepsilon}} \leq \sum_{k=1}^{\infty}\P{Z^N \notin K_{\varepsilon}^k} \leq \varepsilon/2,
\end{align*}
and as $M$ is compact, this shows that $((Z^N,\mathsf{L}^N))_{N\in\N}$ is tight on $\ExtendedE$. 
 Theorem~\ref{thm:ExtendedEmbedding} now shows that $$K:= (\hat K_\varepsilon - r\check M) \cup (\hat K_\varepsilon + r\check M)$$ is compact in $D([-1,\infty))$. We obtain, uniformly in $N \in \N$,

\begin{align*}
\P{\hat X^{N} \notin K} \leq \P{ Z^{N} \notin  K_\varepsilon} + \P{|\alpha_N|>r}  < \varepsilon.
\end{align*}  
So $(\hat X^{N})_{N\in\N}$ is tight on $D([-1,\infty))$.
\end{proof}

When we are concerned with a process admitting a decomposition of the form \eqref{eq:Xdecomposition} in $D([0,\infty))$, from now on we will not distinguish between the process and its extension as given in \eqref{eq:processformhat} to $D([-1,\infty))$ in our notation.

\begin{definition}
Set $X^N:=(X^{1,N},\dots,X^{N,N})$. We say that $X^N$ is $N$-exchangeable, if 
\begin{align*}
\law(X^N)=\law ((X^{\sigma(1),N},X^{\sigma(2),N},\dots,X^{\sigma(N),N})),
\end{align*}
for any permutation $\sigma$ of $\{1,\dots,N\}$.
\end{definition}


\begin{corollary}\label{cor:tightness}
Suppose that $X^N$ satisfies the dynamics
\begin{align*}
X_t^{i,N} = X_{0-}^{i,N} + B_t^{i} - \alpha \mathsf{L}_t^N
\end{align*}
where $\alpha > 0$, $(X_{0-}^N)_{N\in\N}$ are $N$-exchangeable random vectors, $(B_i)_{i\in\N}$ are independent Brownian motions, and
$L^N \in \DistFunctions$. If 
$(X_{0-}^{1,N})_{N\in \N}$ is tight on $\R$, then the empirical measures
\begin{equation}\label{eqn21}
\EmpM_N = \frac{1}{N} \sum_{i=1}^{N} \delta_{X^{i,N}} \qquad\text{and}\qquad
\EmpMext_N = \frac{1}{N} \sum_{i=1}^{N} \delta_{(X_{0-}^{i,N}+B^i,\mathsf{L}^N)}
\end{equation}
are tight on $\mathcal{P}(D([-1,\infty)))$ and $\mathcal{P}(\ExtendedE)$, respectively.
\end{corollary}
\begin{proof}
Fix $\varepsilon > 0$ and let $T>0$. As $C([0,T])$ is a Polish space, there is a compact set $K \subseteq C([0,T])$ such that $ \P{B^1 \in K} > 1 - \varepsilon/2.$
By assumption, there is a compact set $K_{0-} \subseteq \R$ such that
$ \mathbb P(X_{0-}^{1,N} \in K_{0-}) > 1- \varepsilon/2$, for each $N \in \N.$
As $K_{0-}+K$ is compact in $C([0,T])$, and 
$$ \P{X_{0-}^{1,N} + B^1 \notin K_{0-} + K} \leq \varepsilon, 
\quad N \in \N, $$ 
we find that $(X_{0-}^{1,N}+B^1)_{N\in\N}$ is tight on $C([0,T])$ for every $T>0$. The claim now follows from Theorem~\ref{thm:Dinftytightness} and Proposition 2.2 in \cite{sznitman1991topics}.
\end{proof}

We state a technical result for future reference.
\begin{corollary}\label{cor:Phantomcorollary}
 Let $(\EmpM_N)_{N\in \N}$ be given as in Corollary~\ref{cor:tightness}. Then there 
are $\mathcal{P}(\ExtendedE)$-valued random variables $\EmpMext, \EmpMext_N$ such that, after passing to subsequences if necessary,
$
\law(\mu_N) = \law(\iota_{\alpha}(\xi_N)),$
$\EmpMext_N \to \EmpMext,$ and $\iota_{\alpha}(\EmpMext_N) \to \iota_{\alpha}(\EmpMext)$
almost surely. Moreover, 
\begin{equation}\label{eqn10}
\law(\EmpMext_N)=\law\Big(\frac{1}{N} \sum_{i=1}^{N} \delta_{(X_{0-}^{i,N}+B^i,\mathsf{L}^N)}\Big).
\end{equation}
\end{corollary}
\begin{proof}
Set $\EmpMext_N$ as in Corollary~\ref{cor:tightness} and note that $\EmpM_N=\iota_{\alpha}(\EmpMext_N)$ for every $N \in \N$. After passing to subsequences if necessary,
we may assume that $\law(\EmpMext_N) \to \law(\EmpMext)$ by virtue of Corollary~\ref{cor:tightness} for some random variable $\EmpMext$. By the Skorokhod representation theorem, we may assume as well that 
$
\lim_{N\to\infty} \EmpMext_N = \EmpMext
$
 holds in $\mathcal{P}(\ExtendedE)$ for a representation sequence.
Since $\iota_{\alpha}$ is continuous by Theorem~\ref{thm:ExtendedEmbedding}, this implies 
$
\lim_{N\to\infty}\iota_{\alpha}(\EmpMext_N) = \iota_{\alpha}(\EmpMext)
$
almost surely.
\end{proof}

\section{Convergence of solutions}
\label{sec:convergence}
Consider the setting described in Corollary~\ref{cor:tightness} where 
\begin{align*}
X_t^{i,N} = X_{0-}^{i,N} + B_t^{i} - \alpha \mathsf{L}_t^N
\end{align*}
for  $\alpha > 0$, some $N$-exchangeable random vectors $(X_{0-}^N)_{N\in\N}$, some  independent Brownian motions $(B_i)_{i\in\N}$, and some
$\mathsf{L}^N \in \DistFunctions$ adapted to the filtration $(\mathcal F_t^N)_{t\geq0}$ generated by $(X_{0-}^{N},B^1,\ldots,B^N)$. Assume that $(X_{0-}^N)_{N\in\N}$ is independent of $(B_i)_{i\in\N}$
as well as 
$(X_{0-}^{1,N})_{N\in \N}$ is tight on $\R$. Setting again
\begin{equation}\label{eqn20}
\EmpM_N = \frac{1}{N} \sum_{i=1}^{N} \delta_{X^{i,N}}
\end{equation}
 we can then make use of Corollary~\ref{cor:Phantomcorollary} to deduce that there 
are $\mathcal{P}(\ExtendedE)$-valued random variables $\EmpMext, \EmpMext_N$ such that $\EmpMext_N$ satisfies \eqref{eqn10} and, after passing to subsequences if necessary,
\begin{equation}\label{eqn11}
\law(\mu_N) = \law(\iota_{\alpha}(\xi_N)),\quad\EmpMext_N \to \EmpMext,\quad\text{and}\quad\iota_{\alpha}(\EmpMext_N) \to \iota_{\alpha}(\EmpMext)
\end{equation}
almost surely. The goal of this section is now to study the properties of the random measure $\EmpMext$. This analysis will lead to the conclusion that solutions to the particle systems converge to solutions of the McKean--Vlasov problem \eqref{eq:mckeanproblem}.


This result might not seem new, but the fact that we are working with the space $\overline{E}$ embedded into $D([-1,\infty))$ via the map $\iota_{\alpha}$ implies that the involved results in the literature cannot be applied directly. When significant simplifications due to the  different setting are possible, we include them in the (sometimes alternative) proofs

\begin{definition}\label{def:pathfunctionals}
For $t \in \R$ and $x\in D([-1,\infty))$, define the path functionals
$$
\tau_0 (x) := \inf\{s \geq 0: x_s \leq 0\}\quad\text{and}\quad
\ellfunc_t (x) := \ind{[\tau_0 (x) \leq t]}. 
$$
\end{definition}
Considering a sequence of constant positive functions converging to zero shows that $\ellfunc_t$ is not continuous on $D([-1,\infty))$ with the $M_1$-topology. However, it turns out that $\ellfunc_t$ is continuous at paths that satisfy a certain crossing property (going back to \cite{delarue2015particle}), which ensures that the path actually dips below the $x$-axis at the first hitting time of zero (see Lemma \ref{lemma:muellconvergence}).

\begin{lemma}\label{lemma:infrestriction}
Let $x^n,x \in \iota_{\alpha}(\ExtendedE)$ and suppose that $x^n \to x$ in $D([-1,\infty))$. Then, if $t\geq 0$ is any continuity point of $x$, it holds that
\begin{align*}
\lim_{n\to\infty} \inf_{0\leq s \leq t}x_s^n = \inf_{0\leq s\leq t}x_s.
\end{align*}
\begin{proof}
The result follows by Lemma~\ref{lemma:limM1inf} and the fact that for each $y \in \iota_{\alpha}(\ExtendedE)$ and $t \geq 0$ we have  
$
\inf_{-1\leq s\leq t} y_s = \inf_{0\leq s\leq t} y_s
$.
\end{proof}
\end{lemma}

\begin{lemma}\label{lemma:muellconvergence}
Assume that $(\muext^n)_{n\in \N}$ is a convergent sequence of probability measures on $\ExtendedE$ with limit $\muext$. Define $\mu^n := \iota_{\alpha}(\xi^n)$ and $\mu := \iota_{\alpha}(\xi).$ If $\mu$-almost every path satisfies the crossing property, i.e., we have
\begin{align}\label{eq:crossingmu}
\mu \big(\big\{ x \in D([-1,\infty)): \inf_{0\leq s\leq h} (x_{\tau_0+s} - x_{\tau_0}) = 0 \big\}\big) = 0, \quad h > 0, 
\end{align}
then
$
\lim_{n\to\infty} \langle \mu^n, \ellfunc \rangle = \langle \mu, \ellfunc \rangle
$
 holds in $\DistFunctions$.
\begin{proof}
The proof is analogous to the proof of Proposition 5.8 in \cite{delarue2015particle}, making use of Lemma \ref{lemma:infrestriction}.
\end{proof}
\end{lemma}

The next lemma can be proven by translating the last part of the proof of Lemma 5.9 in \cite{delarue2015particle} to the current setting. As our setup allows for a simpler proof using martingale convergence arguments\footnote{Note that in \cite{ledger2021mercy}, martingale convergence arguments are employed to characterize limit points of the particle system with more general dynamics, also allowing for common noise.}, we provide a proof sketch.

\begin{lemma}\label{lem:wBrownian} 
For almost every realization $\omega$, if $\law((W,\mathsf{L}))=\EmpMext(\omega)$, then $W-W_0$ is a Brownian motion with respect to the filtration generated by $(W,\mathsf{L})$. In particular, $W-W_0$ is independent of $W_0$.
\end{lemma}
\begin{proof}[Proof Sketch.]
We employ L\'evy's characterization of Brownian motion. 
Let $g_i \in C_b(\R^2)$ for $i=1,\dots,n$ and
let $0 \leq s_1 < \dots < s_n \leq s < t$, and set $G(w,\ell) := \prod_{i=1}^n g(w_{s_i},\ell_{s_i})$. It holds that
\begin{align*}
\mathbb E\bigg[\left(\int_{\ExtendedE}(w_t-w_s)G(w,\ell)~d\EmpMext(w,\ell)\right)^2\bigg] &= \lim_{N\to\infty} \E{\left(\int_{\ExtendedE}(w_t-w_s)G(w,\ell)~d\EmpMext_N (w,\ell)\right)^2} \\ 
&= \lim_{N\to\infty} \frac{1}{N} \E{\left((B_t^1 - B_s^1) G(X_{0-}^{1,N}+B^{1},\mathsf{L}^N)\right)^2} = 0,
\end{align*}
Using (right-)continuity and the fact that the Borel $\sigma$-algebra on $\ExtendedE$ is generated by the evaluation mappings, 
we conclude that for almost every realization $\omega$, if $\law((W,\mathsf{L}))=\EmpMext(\omega)$, then $W-W_0$ is a
 martingale with respect to the filtration generated by $(W,\mathsf{L})$. Analogously, we can show that $(W_t-W_0)^2 -t$ is a martingale with respect to the filtration generated by $(W,\mathsf{L})$. By L\'evy's characterization of Brownian motion the claim follows.
\end{proof}

The next lemma shows that the limiting measures of particle 
systems such as \eqref{eq:particleN} satisfy the crossing property, 
rendering the loss functional continuous. This result is Lemma 5.9 in \cite{delarue2015particle}.

\begin{lemma}\label{lemma:Crossingpropertyclosed}
Suppose that $\law(\EmpM_N)\to\law(\EmpM)$ for some random variable $\mu$. Then  
$\EmpM$
  satisfies the crossing property \eqref{eq:crossingmu} almost surely.
\end{lemma}
\begin{proof}
Recall that by \eqref{eqn11} we have
$
\law(\EmpM) =\law( \iota_{\alpha}(\EmpMext)).
$

Fixing $h>0$ we can then compute
\begin{align*}
&\mathbb E\big[ \mu\big(\big\{x \in D([-1,\infty)) \colon \inf_{0\leq s\leq h} (x_{\tau_0+s} - x_{\tau_0}) = 0  \big\}\big)\big] \\
&\quad=\mathbb E\big[ \xi\big(\big\{(w,\ell) \in \ExtendedE\colon \inf_{0\leq s\leq h} [(w_{\tau_0+s} - w_{\tau_0}) -\alpha (\ell_{\tau_0+s} - \ell_{\tau_0})] = 0 \big\}\big)\big] \\
&\quad\leq\mathbb E\big[ \xi\big(\big\{(w,\ell) \in \ExtendedE\colon \inf_{0\leq s\leq h} (w_{\tau_0+s} - w_{\tau_0}) = 0 \big\}\big)\big],
\end{align*}
where the inequality is due to the fact that $\ell$ is increasing. 
Observe that for almost every realization $\omega$, if $\law((W,L))=\EmpMext(\omega)$, then 
$\tau_0 = \tau_0(\iota_{\alpha}(W,L))$ is a stopping time with respect to the filtration generated by $(W,L)$. Since  
$W-W_0$ is a Brownian motion with respect to the same filtration by Lemma~\ref{lem:wBrownian},
 the strong Markov property yields
\begin{align*}
&\mathbb E\big[ \xi\big(\big\{(w,\ell) \in \ExtendedE\colon \inf_{0\leq s\leq h} (w_{\tau_0+s} - w_{\tau_0}) = 0 \big\}\big)\big] 
= \mathbb P\big(\inf_{0\leq s \leq h} B_{s} = 0\big)= 0.
\end{align*}
We have shown that 
$$\mu \big(\big\{ x \in D([-1,\infty)): \inf_{0\leq s\leq h} (x_{\tau_0+s} - x_{\tau_0}) = 0 \big\}\big) = 0$$
holds almost surely. Repeating this reasoning for $h=h_n$ for a sequence $(h_n)_{n\in \N}$ such that $h_n > 0$ and $\lim_{n\to\infty} h_n = 0$ yields the result. 
\end{proof}

The next result shows that weak limits of laws pertaining to the particle
system \eqref{eq:particleN} correspond to laws of solution processes to the McKean--Vlasov problem \eqref{eq:mckeanproblem} and was first shown as Theorem 4.4 in \cite{delarue2015particle}. We present an alternative proof.

\begin{proposition}\label{thm:finitedimconvergence} For $N \in \N$, let $(X^N,\L^N)$ be a solution to the particle system 
\begin{equation}\label{eqn6}
\left\{\begin{aligned}
X_{t}^{i,N} &= X_{0-}^{i,
N} + B_{t}^{i} - \L_t^N \\
\tau_{i,N} &= \inf\{t \geq 0: X_{t}^{i,N} \leq 0\} \\
\L_{t}^N &= \frac{\alpha}{N}\sum_{i=1}^{N} \ind{\left[\tau_{i,N}\leq t\right]}, 
\end{aligned}\right.
\end{equation}
and 
let $(\EmpM_N)_{N\in\N}$ denote the corresponding empirical measures \eqref{eqn20}.
Suppose that for some random variable  $\EmpM$ and some measure $\lawXnull \in \mathcal{P}(\R)$ we have
$$
\lim_{N\to\infty} \frac{1}{N}\sum_{i=1}^{N}\delta_{X_{0-}^{i,N}} = \lawXnull,
$$
and  $\law(\EmpM_N) \to \law(\EmpM)$ along some subsequence. Then 
$\mu$ coincides almost surely with the law of a solution process to the
McKean--Vlasov problem \eqref{eq:mckeanproblem} for $\law(X_{0-})=\lawXnull$.
\end{proposition}
\begin{proof} 
Without loss of generality we may  assume that $\EmpM_N = \iota_{\alpha}(\EmpMext_N)$ (and thus \eqref{eqn21}) and set $\EmpM = \iota_{\alpha}(\EmpMext)$.
Observe that the map 
$
t \mapsto \E{\langle \mu, \ellfunc_t \rangle}
$
is increasing, and therefore has at most countably many discontinuities, and the same holds for the map $t \mapsto \E{\int \ell_t ~d\EmpMext(w,\ell)}$. Let $J$ be the set of discontinuities of these maps and fix $t \notin J$. Then   
$
\langle \mu, \ellfunc_{t-} \rangle = \langle \mu, \ellfunc_t \rangle
$
almost surely.
By Lemma \ref{lemma:Crossingpropertyclosed}  we can apply Lemma \ref{lemma:muellconvergence} to obtain that 
\begin{align}\label{eq:muellconvergence}
\lim_{N\to\infty} \langle \EmpM_N, \ellfunc_t \rangle = \langle \EmpM, \ellfunc_t \rangle
\end{align}
holds almost surely. In the next three steps we prove that $\EmpM$ coincides almost surely with the law of a solution process to \eqref{eq:mckeanproblem}. 
\\

\underline{Step 1:} We show that for almost every realization $\omega$, if $\law((W,L))=\EmpMext(\omega)$, then $L\equiv\langle \mu(\omega), \ellfunc \rangle$ almost surely. Note that we have
\begin{align*}
\E{\int_{\ExtendedE}\left|\left|\ell_t - \langle\EmpM,\ellfunc_t\rangle\right|-\left|\ell_t - \langle\EmpM_N,\ellfunc_t\rangle\right|\right|~\mathrm{d}\EmpMext_N (w,\ell)} 
&\leq \E{\left|\langle\EmpM,\ellfunc_t\rangle - \langle\EmpM_N,\ellfunc_t\rangle\right|}
\end{align*}
which vanishes as $N \to \infty$ by \eqref{eq:muellconvergence} and the dominated convergence theorem. By choice of $t \notin J$, the map ${\ell \to \ell_t}$ is continuous from 
$\DistFunctions$ to $\R$ for $\xi$-almost every $\ell$, and it follows that
\begin{align*}
\E{\int_{\ExtendedE}|\ell_t - \langle \EmpM, \ellfunc_t \rangle|~\mathrm{d}\EmpMext(w,\ell)} 
&= \lim_{N\to\infty}\E{\int_{\ExtendedE}|\ell_t - \langle \EmpM, \ellfunc_t \rangle|~\mathrm{d}\EmpMext_N(w,\ell)} \\
&= \lim_{N\to\infty} \E{\int_{\ExtendedE}|\ell_t - \langle \EmpM_N, \ellfunc_t \rangle|~\mathrm{d}\EmpMext_N(w,\ell)}.
\end{align*} 
Since  $\langle \mu_N, \ellfunc \rangle = \ell$ holds for $\EmpMext_N$-almost every $\ell \in \DistFunctions$, almost surely, we can conclude that this expression is equal to 0.
The claim follows by letting $t$ range through a countable dense subset of $[0,\infty)\setminus J$ and using right-continuity.\\

\underline{Step 2:} We show that for almost every realization $\omega$, if $\law((W,\mathsf{L}))=\EmpMext(\omega)$, then $\law(W_{0})=\lawXnull$. To that end, let that 
the evaluation map $\ExtendedEproj_0$ defined on $\ExtendedE$ be given by $(w,\ell) \mapsto w_0$. Since $\ExtendedEproj_0$  is a continuous map, letting $\ExtendedEproj_0 (\EmpMext)$ denote the pushforward of $\EmpMext$ by $\ExtendedEproj_0$, the continuous mapping 
theorem shows
\begin{align*}
\ExtendedEproj_0(\EmpMext) = \lim_{N\to\infty} \ExtendedEproj_0(\EmpMext_N)
= \lim_{N\to\infty} \frac{1}{N}\sum_{i=1}^{N}\delta_{X_{0-}^{i,N}} = \lawXnull.
\end{align*}

\underline{Step 3:} We conclude by showing that 
for almost every realization  $\omega$, if $\law(X)=\EmpM(\omega)$, then $\law(X_{0-})=\lawXnull$ and $X-X_{0-}+\alpha\langle \EmpM(\omega),\ellfunc\rangle$ is a Brownian motion independent of $X_{0-}$.
Since $\iota_{\alpha}(\EmpMext) = \EmpM$, by Step 1 we know that if $\law((W,\mathsf{L}))=\EmpMext(\omega)$, then
$$\law((W_0,W-W_0))=\law((X_{0-},X-X_{0-}+\alpha\langle \EmpM(\omega),\ellfunc\rangle)).$$
The claim now follows by Step 2 and Lemma~\ref{lem:wBrownian}.
\end{proof}

\section{Identification of the limit}
\label{sec:identification}
Having proved convergence of the particle system to solutions of the McKean--Vlasov problem, we are now interested in characterizing the limit points of minimal and physical solutions of the particle system.
We shed some light on this intricate topic by either perturbing the initial condition of the particle system (Section~\ref{sec61}) or of the McKean--Vlasov problem (Section~\ref{sec:62}).

\subsection{The perturbed particle system}\label{sec61}
The next result is a key technical result for the theorems in this section.

\begin{proposition}\label{thm:optproblemconv}
Fix $\alpha>0$, $\gamma \in (0,1/2)$, and let $\Lpertmin^{N}$ denote the perturbed minimal solution to the particle system, which for $N \in \N$ is
the minimal solution to the perturbed particle system
\begin{equation}\label{eq:pert_particleN}
\left\{\begin{aligned}
X_{t}^{i,N} &= X_{0-}^{i} + N^{-\gamma} + B_{t}^{i} - \L_t^N \\
\tau_{i,N} &= \inf\{t \geq 0: X_{t}^{i,N} \leq 0\} \\
\L_{t}^N &= \frac{\alpha}{N}\sum_{i=1}^{N} \ind{\left[\tau_{i,N}\leq t\right]}.
\end{aligned}\right.
\end{equation}
If $\Linfmin$ is the minimal solution to the McKean--Vlasov problem \eqref{eq:mckeanproblem}, then
\begin{equation}\label{eq:pertasymptoticinequality}
\limsup_{N\to\infty} \E{\Lpertmin^{N}_t } \leq \Linfmin_t,
\end{equation}
holds for every $t > 0$.
\end{proposition}
\begin{proof}
The key idea of this proof is to compare both $\Linfmin$ and $(\Lpertmin^N)_{N\in \N}$ to a sequence of optimizers 
of suitable optimization problems. To that end, let $\varepsilon > 0$ be such that $\beta := \gamma + \varepsilon < 1/2$. 
Set $\PErandom$ to be the set of all random variables $\mu : \Omega \rightarrow \mathcal{P}(E)$. Fix $t>0$ and define the cost functional
$$
c_N(\mu) := \langle \mu,\ellfunc_t \rangle + N^{\beta} \|\Gamma_N[\langle \mu,\ellfunc\rangle] - \langle \mu,\ellfunc\rangle\|_{\infty},
$$
where $\Gamma_N$ is the operator introduced in \eqref{eqdef:GammaN} and $\|\cdot\|_{\infty}$ is the uniform norm on $[0,\infty)$.
For $t \geq 0$, let $V_t^N$ be the optimal value 
\begin{equation}
\label{eq:optn}
V_t^N := \inf_{\mu \in \PErandom} \E{c_N(\mu)}.
\end{equation}
The definition \eqref{eq:optn} depends on the underlying probability space $(\Omega,\mathcal{F},\mathbb{P})$. However, our comparison argument works for any choice of $(\Omega,\mathcal{F},\mathbb{P})$ as long as $\EmpM_N \in \mathcal{R}$ for all $N \in \N$ and for any $\mu \in \mathcal{R}$, the map $\omega \mapsto c_N(\mu(\omega))$ is measurable. As \eqref{eq:pertasymptoticinequality} does not depend on the choice of probability space, it is sufficient to carry out the argument for a specific choice of $(\Omega,\mathcal{F},\mathbb{P})$. For the construction of a probability space satisfying the required properties see Example~\ref{ex1} in Section~\ref{appC} in the appendix.\\

 \underline{Step 1:} We show that asymptotically, $\Linfmin$ is an upper bound for $\alpha V_t^N$. Set $\Lmin_t := \frac{1}{\alpha}\Linfmin_t$. Taking $\mu$ to be constant and equal to the law of the minimal solution process of \eqref{eq:mckeanproblem}, that is $\mu(\omega) = \lawmin$, we find
$$
V_t^N \leq \Lmin_t + N^{\beta} \E{\|\Gamma_N[\Lmin] - \Lmin\|_{\infty}}.
$$
Note that by definition, $\Gamma_N [\Lmin]$ is the $N$-th empirical distribution function associated to the iid sequence $\tau^1[\Lmin],\tau^2 [\Lmin],\dots$, where
$$
\tau^{i}[\Lmin] := \inf\{t \geq 0: X_{0-}^{i}+B_t^{i} - \Linfmin_t \leq 0\}, \quad i \in \N, 
$$
and that $\P{\tau^{i}[\Lmin]\leq t}=\Lmin_t$ for each $t$.
By the 
Dvoretzky-Kiefer-Wolfowitz inequality (cf Corollary 1 in \cite{massart1990tight}), it holds for $x > 0$ that
$$
\P{\sqrt{N}\|\Gamma_N[\Lmin] - \Lmin\|_{\infty} > x} \leq 2\exp(-2x^2).
$$
It therefore follows that
\begin{align*}
 \E{N^{\beta}\|\Gamma_N[\Lmin] - \Lmin\|_{\infty}} &= \int_{0}^{\infty} \P{N^{\beta}\|\Gamma_N[\Lmin] - \Lmin\|_{\infty} > x}~\mathrm{d}x \\
 &= \int_{0}^{\infty} \P{N^{\nicefrac{1}{2}}\|\Gamma_N[\Lmin] - \Lmin\|_{\infty} > N^{\nicefrac{1}{2}-\beta}x}~\mathrm{d}x \\
 &\leq 2 \int_{0}^{\infty}\exp(-2N^{1-2\beta}x^2)~\mathrm{d}x \\
 &= \sqrt{\frac{\pi}{2}}N^{\beta -\nicefrac{1}{2}}.
\end{align*}
As $\beta < 1/2$, this quantity vanishes as $N \to \infty$, which proves
$$
\limsup_{N\to\infty} \alpha V_t^N \leq \Linfmin_t.
$$
\underline{Step 2:} We show that asymptotically, $\limsup_{n\to\infty}\mathbb E[\Lpertmin_t^N]\leq \limsup_{n\to\infty}\alpha V_t^N$. Consider a sequence of $\mathcal{P}(E)$-valued random variables $(\nearopt_N)_{N\in \N}$ such that
$$
\E{c_N(\nearopt_N)} - V_t^N \leq \frac{1}{N}, \quad N \in \N.
$$
Plugging the empirical measure associated to the solution of the
$N$-particle system into the objective function in \eqref{eq:optn} yields $V_t^N \leq 1$ and consequently
\begin{equation}\label{eq:Enearoptbound}
\E{N^{\beta} \|\Gamma_N[\langle \nearopt_N,\ellfunc\rangle] - \langle \nearopt_N,\ellfunc\rangle\|_{\infty}} \leq 1+\frac{1}{N}, \quad N \in \N.
\end{equation}
Let $A_N := [\|\Gamma_N[\langle \nearopt_N,\ellfunc\rangle] - \langle \nearopt_N,\ellfunc\rangle\|_{\infty} \leq N^{-\gamma}]$, then the Markov inequality together with \eqref{eq:Enearoptbound} yield
\begin{equation}\label{eq:PANcbound}
\P{A_N^c} \leq N^{\gamma-\beta}(1 + \frac{1}{N}) \leq 2N^{-\varepsilon}.
\end{equation}
For the last step, in analogy to \eqref{eqdef:GammaN}, define the operator $\GammaNpert$ via
$$
\left\{\begin{aligned}
\Xpert_{t}^{i,N}[\ell] &:= X_{0-}^{i} + \alpha N^{-\gamma} + B_{t}^{i} - \alpha \ell_t \\
\taupert_{i,N}[\ell] &:= \inf\{t \geq 0: \Xpert_{t}^{i,N}[\ell] \leq 0\} \\
\GammaNpert[\ell]_t &:= \frac{1}{N}\sum_{i=1}^{N} \ind{\left[\taupert_{i,N}[\ell]\leq t\right]}.
\end{aligned}\right.
$$
On the set $A_N$, we have by definition
\begin{equation}\label{eq:ANinequality}
\langle \nearopt_N, \ellfunc_t \rangle \geq \Gamma_N [\langle \nearopt_N, \ellfunc\rangle]_t - N^{-\gamma}, \quad t \geq 0.
\end{equation}
As $\langle \nearopt_N,\ellfunc\rangle$ is nonnegative and $\Gamma_N$ is monotone, this implies for all $t \geq 0$
\begin{align*}
\langle \nearopt_N, \ellfunc_t \rangle \geq \Gamma_N [-N^{-\gamma}]_t - N^{-\gamma} = \GammaNpert[0]_t-N^{-\gamma}
\end{align*}
Using the monotonicity of $\Gamma_N$ and applying \eqref{eq:ANinequality} again, this leads to
\begin{align*}
\langle \nearopt_N, \ellfunc_t \rangle \geq \Gamma_N[\GammaNpert[0]-N^{-\gamma}]_t-N^{-\gamma} = \GammaNpert^{(2)}[0]_t - N^{-\gamma}
\end{align*}
for all $t \geq 0$. A straightforward induction shows that
\begin{align*}
\langle \nearopt_N, \ellfunc_t \rangle \geq \GammaNpert^{(k)}[0]_t - N^{-\gamma}, \quad t \geq 0
\end{align*}
holds for all $k \in \N$ on the event $A_N$. Lemma~\ref{lem:miniteration} shows that $\alpha\GammaNpert^{(N)}[0] = \Lpertmin^N$
and so, finally we obtain that on $A_N$ we have
\begin{align}\label{eq:Lpertmuellinequality}
\alpha \langle \nearopt_N, \ellfunc\rangle \geq \Lpertmin^{N} - \alpha N^{-\gamma}.
\end{align}
As $\Lpertmin^{N}$ is bounded by $\alpha$, we find
\begin{align}\label{eq:LpertANasy}
0\leq\mathbb E\big[\Lpertmin^N\big] - \mathbb E\big[\Lpertmin^N\ind{A_N}\big] = \mathbb E\big[\Lpertmin^N\ind{A_N^c}\big] \leq \alpha\P{A_N^c}
\end{align}
which goes to zero as $N \to \infty$ by \eqref{eq:PANcbound}. Conditions
\eqref{eq:LpertANasy} and \eqref{eq:Lpertmuellinequality}  yield
\begin{align*}
\limsup_{N\to\infty} \mathbb E\big[\Lpertmin^{N}_t\big] &= \limsup_{N\to\infty} \mathbb E\big[\Lpertmin^{N}_t \ind{A_N}\big] \leq \limsup_{N\to\infty}\mathbb E\big[\alpha\langle\nearopt_N,\ellfunc_t \rangle\ind{A_N}\big]  \\
&\leq  \limsup_{N\to\infty}\alpha\E{c_N(\nearopt_N)} = \limsup_{N\to\infty} \alpha V_t^N.
\end{align*}
Combining Step 1 and Step 2 then yields the result.
\end{proof}

By virtue of Proposition~\ref{thm:optproblemconv}, we are now in a position to 
prove a propagation of chaos result for the perturbed particle system.
We obtain this result without imposing any restrictions on the distribution
of the initial condition $X_{0-}$.

\begin{theorem}[Propagation of chaos, perturbed]\label{thm:pertpropagation}
Fix $\gamma \in (0,1/2)$ and define $(\Xpertmin^N,\Lpertmin^N)$ to be the 
minimal solution to the perturbed particle system \eqref{eq:pert_particleN}. Let $\lawmin$ be the law of the minimal solution process $\Xmin$ to the
McKean--Vlasov problem \eqref{eq:mckeanproblem}. Then, it holds that
$$
\lim_{N\to\infty} \frac{1}{N}\sum_{i=1}^{N}\delta_{\Xpertmin^{i,N}} = \lawmin
$$
in probability on $\mathcal{P}(D([-1,\infty))$. Furthermore, the sequence of loss functions $(\Lpertmin^N)_{N\in \N}$ converges to $\Linfmin$ in probability with respect to the L\'evy-metric, i.e., for every $\varepsilon > 0$ it holds that
$$
\lim_{N\to\infty} \P{d_L (\Lpertmin^N,\Linfmin) > \varepsilon} = 0,
$$
where $d_L$ denotes the L\'evy-metric.
\end{theorem}
\begin{proof}
We follow the classical sequence of arguments, showing tightness of the empirical measures first, then their convergence to the law of a solution's process to the McKean--Vlasov problem, and finally we check that the corresponding solution is minimal. \\

\underline{Step 1:} Since the sequence $(X_{0-}^{1} + N^{-\gamma})_{N\in\N}$ is tight,  tightness of the empirical measures follows from Corollary~\ref{cor:tightness}.\\

\underline{Step 2:} Let $\EmpMpert_N$ be the empirical measures associated to the minimal solution to \eqref{eq:pert_particleN}. Then by Step 1 there is a random variable $\EmpM$ such that, after passing to a subsequence if necessary, $\law(\EmpMpert_N) \to \law(\EmpM)$. Moreover, because of Varadarajan's theorem (cf \cite[Theorem 11.4.1]{dudley2018real}) and the fact that uniformly continuous functions are convergence determining (cf. Proposition 3.4.4 in \cite{ethier2009markov}) we can establish that 
$$
\lim_{N\to\infty} \frac{1}{N} \sum_{i=1}^{N} \delta_{ X_{0-}^{i}+N^{-\gamma}} = \lawXnull
$$
holds almost surely in $\mathcal{P}(\R)$.
Proposition~\ref{thm:finitedimconvergence} then shows that  $\mu$ coincides almost surely with the law of a solution process to the McKean--Vlasov problem \eqref{eq:mckeanproblem}. \\

\underline{Step 3:} We now  show that $\mu=\lawmin$ almost surely, i.e.~that $\mu$ coincides with law of the minimal solution process to \eqref{eq:mckeanproblem}. 
Let $J \subseteq [0,\infty)$ be the countable set of discontinuities 
of the increasing map 
$
t \mapsto \E{ \langle \mu, \ellfunc_t \rangle}
$ and fix $t\notin J$.
 Then, arguing as in the proof of Proposition~\ref{thm:finitedimconvergence}, we obtain that 
$
\lim_{N\to\infty}\langle \EmpMpert_N, \ellfunc_t \rangle = \langle \mu, \ellfunc_t \rangle
$
almost surely. To be precise the convergence holds for a subsequence of a representative sequence of $(\EmpMpert_N)_{n\in\N}$ but we can assume without loss of generality that it coincides with the original one. Letting $D$ be a countable, dense subset of $[0,\infty)$ with $D \cap J = \emptyset$, 
we then have that
\begin{equation}\label{eqn7}
\lim_{N\to\infty}\langle \EmpMpert_N, \ellfunc_t \rangle = \langle \mu, \ellfunc_t \rangle, \quad t \in D
\end{equation}
holds almost surely.
By Step 2 and the definition of minimal solution, we have $\langle \mu, \ellfunc \rangle \geq \Linfmin$ almost surely. The dominated convergence theorem and Proposition~\ref{thm:optproblemconv} imply
\begin{align*}
 \E{\alpha\langle \mu, \ellfunc_t\rangle} = \lim_{N\to\infty}\E{\alpha\langle \EmpMpert_N, \ellfunc_t\rangle} = \lim_{N\to\infty} \mathbb E\big[\Lpertmin_t^N\big] \leq \Linfmin_t, \quad t \in D.
\end{align*}
We conclude that $\P{\Linfmin_t = \alpha\langle \mu, \ellfunc_t \rangle, ~t \in D} = 1$ and therefore we have $\Linfmin = \alpha\langle \mu, \ellfunc \rangle$ almost surely by right-continuity. It follows that
$
\mu = \lawmin$ almost surely.
\\

 \underline{Step 4:} By \eqref{eqn7} and right continuity we know that $\law(\Lpertmin^N) \to \law(\alpha\langle \mu, \ellfunc \rangle)$. Since $\law(\langle \mu, \ellfunc \rangle)=\delta_{\Linfmin/\alpha}$ by Step 3, we can conclude that $\Lpertmin^N \to \Linfmin$
 in probability on $\DistFunctions$.  
\end{proof}

 Similarly to the situation in the particle system, it turns out that physical solutions to the McKean--Vlasov problem have minimal jumps, which is the content of the following proposition.

\begin{proposition}\label{thm:physjumpinequality}
Suppose $(X,\tau,\L)$ solves \eqref{eq:mckeanproblem}. Then it holds that
\begin{align*}
\Delta \L_t \geq \alpha\inf\{x > 0 \colon \P{\tau \geq t, X_{t-} \in [0,\alpha x]} < x \}
\end{align*}
for any $t \geq 0$.
\end{proposition}
\begin{proof}
See Proposition 1.2 in \cite{hambly2019mckean}.
\end{proof}

It has been established in previous works that weak limits of physical solutions of the particle system \eqref{eq:particleN}
correspond to physical solutions to the McKean--Vlasov problem \eqref{eq:mckeanproblem}. Recall that we call a solution
$(X,\L)$ to \eqref{eq:mckeanproblem} physical, if it satisfies the physical jump condition \eqref{eqdef:physicaljump}.
\begin{theorem}[Physical solutions converge to physical solutions]
\label{thm:physconvergetophys}
Let $(\Xphys^N,\Lphys^N)$ be a physical solution to the particle system \eqref{eqn6}, where $X_{0-}^{i,N} = X_{0-}^{i}+a_N$ for some deterministic sequence $(a_N)_{N\in \N}$ converging to $0$ and some iid sequence $(X_{0-}^i)_{i\in\N}$ such that ${\mathbb E[X^i_{0-}] < \infty}$. Then, if for some $\Linfphys \in \DistFunctions$ we almost surely have $\Lphys^N \to \Linfphys$ along some subsequence  in $\DistFunctions$, it follows that $\Linfphys$ is a physical
solution to the McKean--Vlasov problem \eqref{eq:mckeanproblem}.
\end{theorem}

\begin{proof}
See Section~\ref{appB} in the appendix.
\end{proof}

The next theorem gives a positive answer to a conjecture of Delarue, Nadtochiy and Shkolnikov.

\begin{theorem}\label{thm:minimalisphysical} Suppose that $\E{X_{0-}} < \infty.$ Then, the minimal solution $(\Xmin, \Linfmin)$ to the McKean--Vlasov problem \eqref{eq:mckeanproblem} is physical.

\end{theorem}
\begin{proof}
Let $(\Xpertmin^N,\Lpertmin^N)$ be the minimal solution to the perturbed particle system \eqref{eq:pert_particleN}. By Theorem~\ref{thm:pertpropagation} we know that $\Lpertmin^N \to \Linfmin$ in probability on $\DistFunctions$. We also know by Lemma~\ref{lemma:physicalminimal} that $\Lpertmin^N$
is a physical solution to the perturbed particle system. After passing to a 
subsequence if necessary, we may assume that $\Lpertmin^N \to \Linfmin$ almost surely on $\DistFunctions$. Theorem~\ref{thm:physconvergetophys} then shows that $\Linfmin$ is a physical solution to the McKean--Vlasov problem \eqref{eq:mckeanproblem}. 
\end{proof}

\begin{theorem}[Propagation of minimality]\label{cor:propmin}
Suppose that $\E{X_{0-}} < \infty $ and that the physical solution to the McKean--Vlasov problem \eqref{eq:mckeanproblem} is unique. Denote by $(\Xmin^N, \Linfmin^N)$ the 
minimal solution to the particle system \eqref{eq:particleN}. Let $\lawmin$ be the law of the minimal solution process $\Xmin$ to the
McKean--Vlasov problem \eqref{eq:mckeanproblem}. Then, it holds that
\begin{align}\label{eqn2}
\lim_{N\to\infty} \frac{1}{N}\sum_{i=1}^{N}\delta_{\Xmin^{i,N}} = \lawmin
\end{align}
in probability on $\mathcal{P}(D([-1,\infty))$. Furthermore, the sequence of loss functions $(\Linfmin^N)_{N\in \N}$ converges to $\Linfmin$ in probability with respect to the L\'evy-metric, i.e., for every $\varepsilon > 0$ it holds that
$$
\lim_{N\to\infty} \P{d_L (\Linfmin^N,\Linfmin) > \varepsilon} = 0,
$$
where $d_L$ denotes the L\'evy-metric.
\end{theorem}
\begin{proof}
We argue again in the classical way. Tightness of the empirical measures \eqref{eqn2} follows from Corollary~\ref{cor:tightness} and their convergence to the law of a solution process of the McKean--Vlasov problem \eqref{eq:mckeanproblem} follows by Proposition~\ref{thm:finitedimconvergence}. We know from Lemma~\ref{lemma:physicalminimal} that minimal solutions of the particle system are physical, which allows us to apply Theorem~\ref{thm:physconvergetophys}, telling us that the limit must be the unique physical solution to the McKean--Vlasov problem. By Theorem~\ref{thm:minimalisphysical}, the physical solution must be equal to the minimal solution.
\end{proof}

\begin{remark}
By \cite[Theorem 1.4]{delarue2019global}, if the initial condition $X_{0-}$ admits a bounded Lebesgue density that changes monotonicity finitely often on compact subsets of $[0,\infty)$, the physical solution to the McKean--Vlasov problem \eqref{eq:mckeanproblem} is unique.
\end{remark}

\subsection{The perturbed McKean--Vlasov problem}\label{sec:62}

A modification of the proof of Proposition \ref{thm:optproblemconv}
allows us to ``shift'' the perturbation from the particle system to the 
McKean--Vlasov problem.

\begin{proposition}\label{prop:optprobconvmckean}
Fix $\alpha>0$, $\gamma \in (0,1/2)$, and for $x\in \R$ let  $\Linfmin(x)$ denote the minimal solution of the perturbed McKean--Vlasov problem
\begin{equation}\label{eq:mckeanproblempert}
\left\{
\begin{aligned}
X_t &= X_{0-} -x+ B_t - \L_t \\ 
\tau &= \inf\{t \geq 0: X_t \leq 0 \} \\
\Lambda_t &=\alpha\P{\tau \leq t}.
\end{aligned}\right.
\end{equation}
If $\Linfmin^{N}$ is the minimal solution to the particle system \eqref{eq:particleN}, then
\begin{equation}\label{eq:pertasymptoticinequality2}
\limsup_{N\to\infty} \E{\Linfmin^{N}_t } \leq \limsup_{N\to\infty} \Linfmin_t( N^{-\gamma}),
\end{equation}
for every $t > 0$.
\end{proposition}
\begin{proof}
The proof is analogous to the proof of Proposition \ref{thm:optproblemconv}, only that now we consider the sequence of cost functionals
\begin{align*}
c_N(\mu) := \langle \mu,\ellfunc_t \rangle + N^{\beta} \|\Gamma_N[\langle \mu,\ellfunc\rangle+N^{-\gamma}] - \langle \mu,\ellfunc\rangle\|_{\infty}
\end{align*}
where $\beta = \gamma + \varepsilon$.
\end{proof}

Shifting the perturbation to the McKean--Vlasov problem allows us to show
that propagation of minimality holds true for Lebesgue almost every (fixed) additive perturbation of the initial condition.

\begin{theorem}[Propagation of chaos, almost everywhere]\label{thm:almostpropagation}
For $x \in \R$, let $(\Xmin^N(x),\Linfmin^N(x))$ be the 
minimal solution to the particle system
$$
\left\{\begin{aligned}
X_{t}^{i,N} &= X_{0-}^{i}-x + B_{t}^{i} - \L_t^N \\
\tau_{i,N} &= \inf\{t \geq 0: X_{t}^{i,N} \leq 0\} \\
\L_{t}^N &= \frac{\alpha}{N}\sum_{i=1}^{N} \ind{\left[\tau_{i,N}\leq t\right]},
\end{aligned}\right.
$$
 Furthermore, let $\lawmin(x)$ be the law of the minimal solution process $\Xmin(x)$ to the perturbed
McKean--Vlasov problem \eqref{eq:mckeanproblempert}. Then, there is a co-countable set $D\subseteq \R$ such that
$$
\lim_{N\to\infty} \frac{1}{N}\sum_{i=1}^{N}\delta_{\Xmin^{i,N}(x)} = \lawmin(x)
$$
in probability on $\mathcal{P}(D([-1,\infty))$ for $x \in D$. Furthermore, the sequence of loss functions $(\Linfmin^N(x))_{N\in \N}$ converges to $\Linfmin(x)$ in probability with respect to the L\'evy-metric for every $x \in D$, i.e., for every $\varepsilon > 0$ it holds that
$$
\lim_{N\to\infty} \P{d_L (\Linfmin^N(x),\Linfmin(x)) > \varepsilon} = 0,
$$
where $d_L$ denotes the L\'evy-metric.
\end{theorem}
\begin{proof}
Let $D_0$ be a countable dense subset of $[0,\infty)$. Fix $t \in D_0$ and note that the map $x \mapsto \Linfmin_t(x)$ is increasing and therefore has at most countably many discontinuities. Let the set of all such discontinuities be denoted by $J_t$. Then, $J:= \bigcup_{t\in D_0} J_t$ is countable, and we set $D:= [0,\infty)\setminus J$. Then, for all $x \in D$ it follows from Proposition \ref{prop:optprobconvmckean} that we have
\begin{align*}
\limsup_{N\to\infty} \E{\Linfmin_t^N(x)} \leq \limsup_{N\to\infty}\Linfmin_t(x+ N^{-\gamma}) = \Linfmin_t(x), \quad t \in D_0.
\end{align*}
The remainder of the proof is analogous to the proof of Theorem \ref{thm:pertpropagation}.
\end{proof}

Following the proof of Theorem~\ref{thm:almostpropagation} we can see that if we would have stability of the minimal solution to the McKean--Vlasov problem under additive perturbations of the initial condition, we would obtain propagation of minimality as in Theorem \ref{cor:propmin} 
without having to assume that the physical solution to the McKean--Vlasov
problem is unique. We conjecture that such
a stability result holds true.

\begin{conjecture}
The map $x \mapsto \Linfmin(x)$ is continuous from $\R$ to $\DistFunctions$.
\end{conjecture}

\appendix
\appendixpage

\section{On the $M_1$- and $J_1$-topologies}\label{appA}
\subsection{Why the $J_1$-topology is ill-suited to the problem}
Of the four topologies initally proposed by Skorokhod, the $J_1$-topology is the most popular, and often simply referred to as ``the Skorokhod topology''. 
However, for the present purpose, the $J_1$-topology seems to be simply too strong - in particular, local accumulations of small jumps can obstruct convergence in 
the $J_1$-topology. We illustrate this point with an example, for which we need the following theorem.

\begin{theorem}\label{thm:J1conv}
Let $\ell^n, \ell$ be increasing \cadlag functions on $[0,\infty)$. Then, $\ell^n \to \ell$ as $n \to \infty$ in the $J_1$-topology if and only if there is a dense subset $D \subseteq [0,\infty)$ consisting of continuity points of $\ell$ such that for all $t \in D$ 
\begin{enumerate}[(i)]
\item $\ell_t^n \to \ell_t$
\item $\sum_{s\leq t} |\Delta \ell_s^n|^2 \to \sum_{s\leq t} |\Delta \ell_s|^2$.
\end{enumerate}
\end{theorem}
\begin{proof}
See Theorem VI.2.15 in \cite{jacod2013limit}
\end{proof}

\begin{example}\label{ex:J1notconv}
For $n \in \N$, let $\ell^n$ be an increasing step function such that $\ell^n$ has $n$ jumps of size $1/n$ in the interval $[1/2-1/n,1/2+1/n]$ and is constant otherwise. Interpreting $\ell_t^n$ as the proportion of banks that defaults up to time $t$, this would mean that for all $n \in \N$, all of the banks in the system default after time $1/2-1/n$ and before time $1/2+1/n$. Certainly, as $n$ goes to infinity, we would expect the limit to be the function $\ell_t = \ind{[1/2,\infty]}(t)$, which corresponds to all the banks defaulting at time $1/2$. Assuming that $\ell^n$ converges to some function $g$ in the $J_1$-topology, condition $(i)$ in Theorem~\ref{thm:J1conv} yields that $g(t) =\ind{[1/2,\infty]}(t)$. However, as $\sum_{s\leq t} |\Delta  \ell^n_s|^2 = \frac{1}{n}$, condition $(ii)$ of Theorem~\ref{thm:J1conv} yields that $g$ must be continuous, a contradiction.
\end{example}

\begin{figure}[!htb]\centering
\minipage{0.32\textwidth}
  \includegraphics[width=\linewidth]{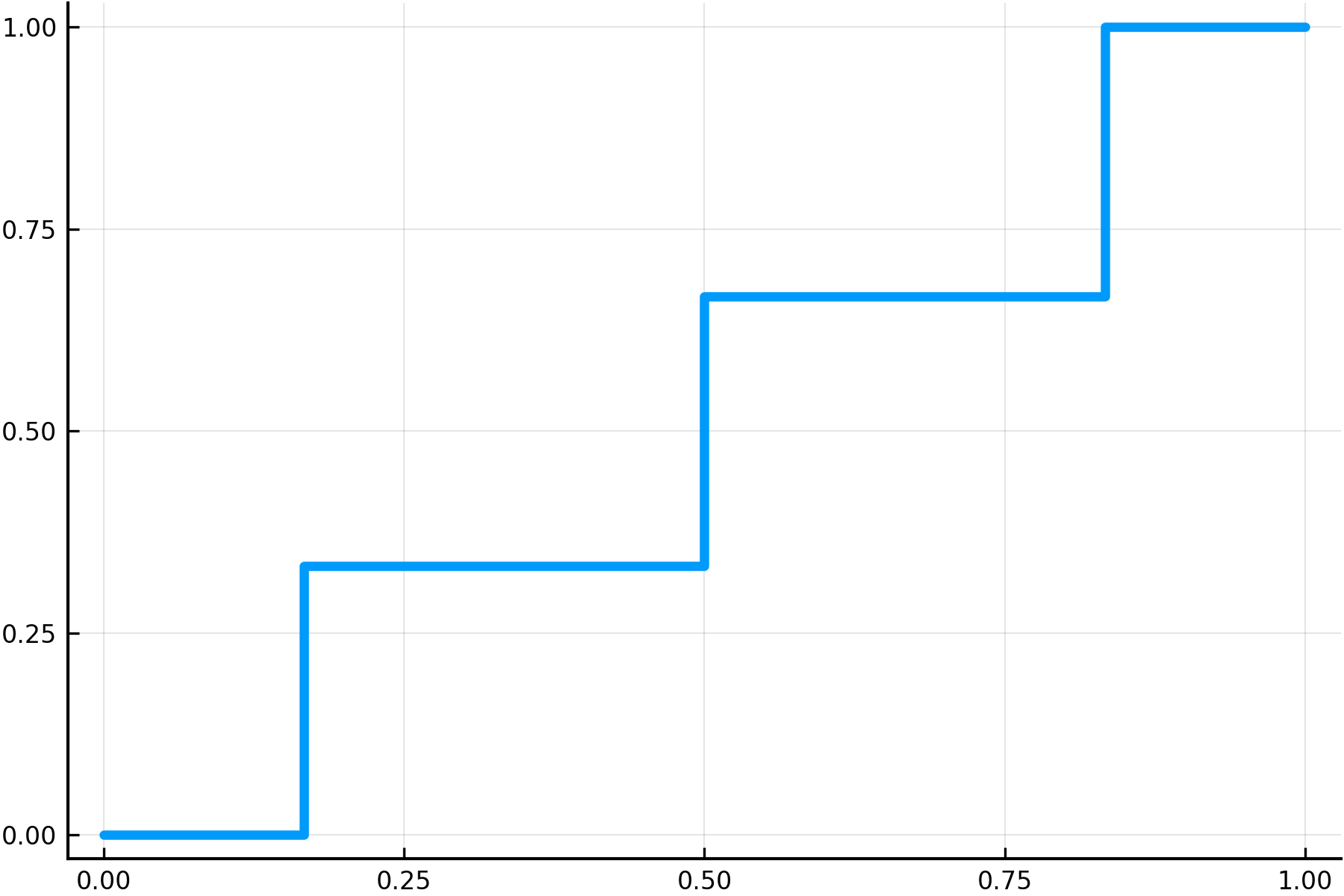}
\endminipage\hfill
\minipage{0.32\textwidth}
  \includegraphics[width=\linewidth]{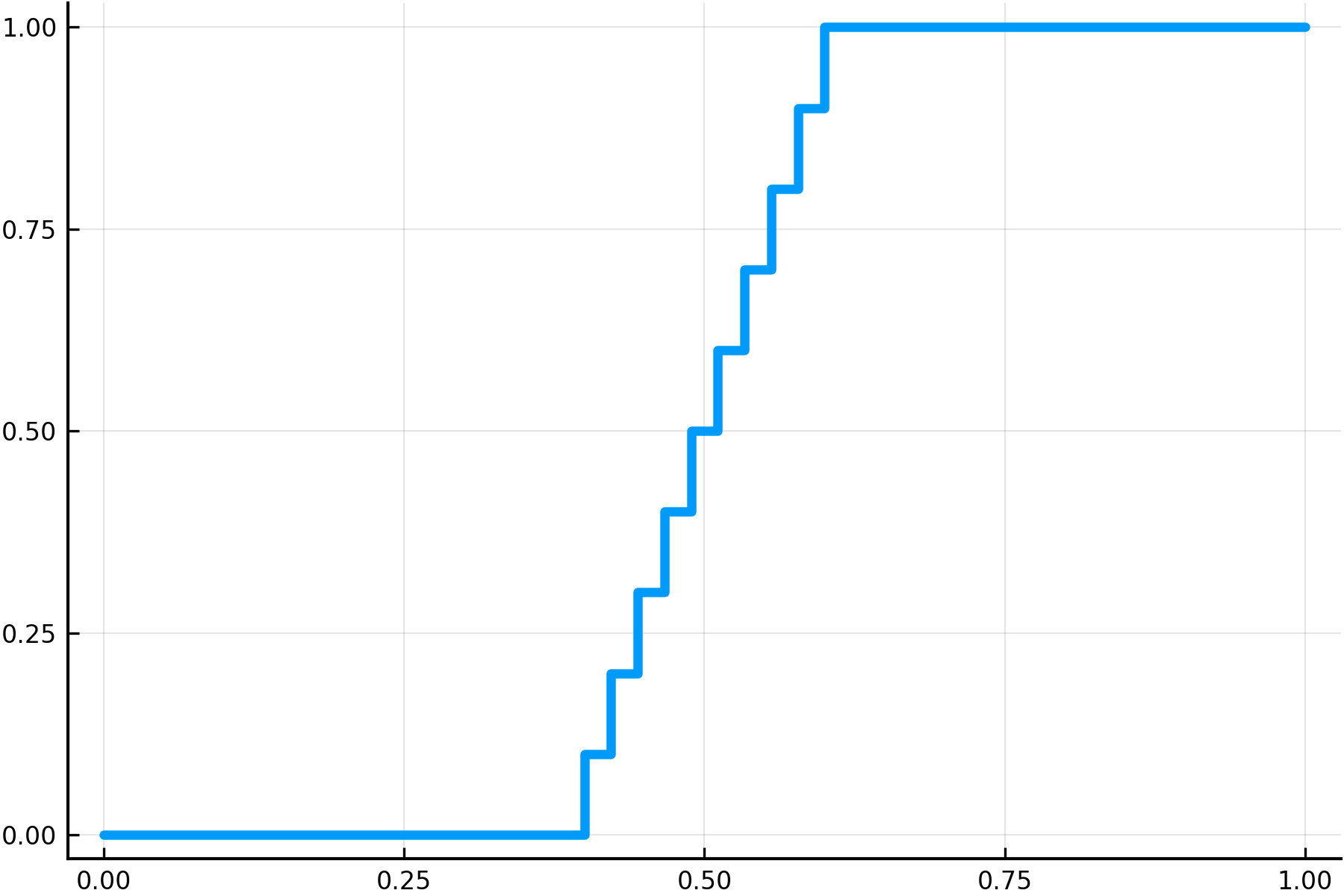}
\endminipage\hfill
\minipage{0.32\textwidth}%
  \includegraphics[width=\linewidth]{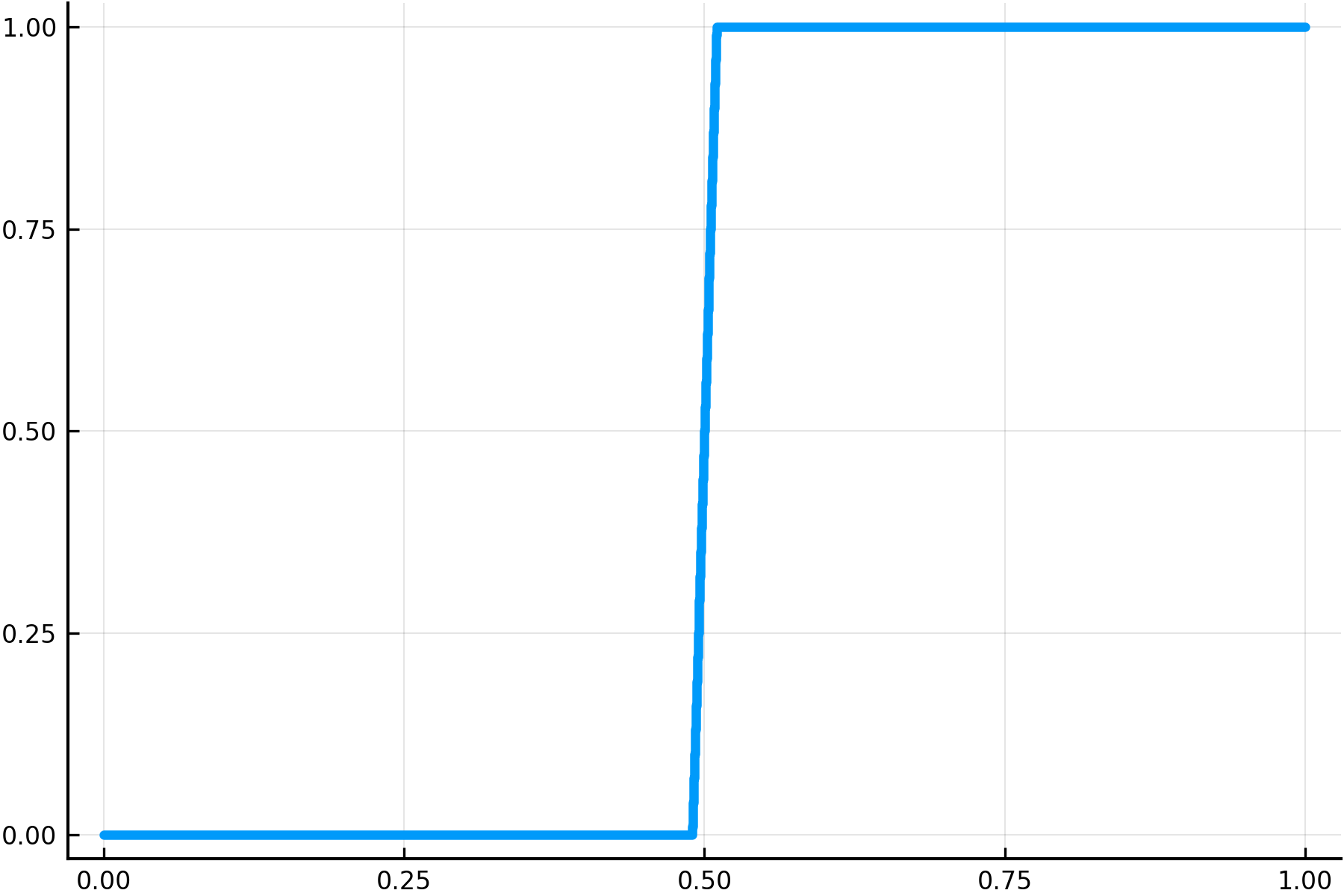}
\endminipage
\caption{The functions $\ell^n$ as defined in Example~\ref{ex:J1notconv} for $n=3,10,100$.}
\label{fig:J1nogood}
\end{figure}

Roughly speaking, the $J_1$-topology allows for some flexibility in the location of jumps in convergent sequences,
while requiring that the size of the jumps in the approximating sequence remains close to the size of the jumps in the limit in a certain sense.
With this intuition in mind, it is not surprising that the space of continuous functions is closed in the space of \cadlag functions endowed with
the $J_1$-topology. In contrast, continuous functions are dense in the space of \cadlag functions endowed with the $M_1$-topology, and the sequence
given in Example~\ref{ex:J1notconv} is convergent in the $M_1$-topology. 

\subsection{Some properties of the $M_1$-topology}
The $M_1$-topology is strictly weaker than the $J_1$-topology (Example~\ref{ex:J1notconv} serves as an example of a sequence which converges in $M_1$ but not in $J_1$ in view of Lemma~\ref{thm:M1monconvergence}). This is not necessarily a weakness: whenever we want to show tightness, a weaker topology is favorable because the conditions for compactness are less strict. We will see that the $M_1$-topology is particularly well-suited to deal with monotone functions in this context. 

We mention some fundamental properties of the $M_1$-topology in the following. For a path $x \in D([T_0,T])$, the space of c\`adl\`ag paths defined on $[T_0,T]$ taking values in $\R$ that are left-continuous at $T$, we define the completed graph 
\begin{align*}
\Graph_x := \{(t,y) \in [T_0,T]\times\R : y\in [x_{t-},x_t]\},
\end{align*}
where $[x_{t-},x_t]$ is the non-ordered segment between $x_{t-}$ and $x_t$ (this takes into account that $x_{t-}$ might be larger than $x_t$). The set $\Graph_x$ can be ordered in the following way: for $(t_1,y_1),(t_2,y_2) \in \Graph_x$, we say that $(t_1,y_1) \preceq (t_2,y_2)$ if either $t_1 < t_2$, or $t_1 = t_2$ and $|x_{t_{1}-}-y_1| \leq |x_{t_{2}-} - y_2|$. This order can be conceptualized more easily in the following way. The completion $\Graph_x$ can be imagined as a path in $2$-dimensional space, where we complete the graph of $x$ by connecting the discontinuities with straight lines going through the points $(t,x_{t-})$ and $(t,x_t)$. We can imagine a particle traveling along $\Graph_x$ from left to right; for $A,B \in \Graph_x$, we have $A\preceq B$ iff the particle reaches $A$ before it reaches $B$. This is illustrated in Figure~\ref{fig:graphorder}.

\begin{figure}\label{fig:graphorder}
\centering
\begin{tikzpicture}[domain=0:4]
    \draw[->] (0,0) -- (4.2,0) node[below,midway]{a)} node[right] {$t$};
    \draw[->] (0,0) -- (0,4.2) node[above] {$f(t)$};
    \draw[thick,color=cyan]  (0,1) -- (1,1);
    \draw[thick,color=cyan]  (1,3) -- (3,3);
    \draw[thick,color=cyan]  (3,2) -- (4,2);
\end{tikzpicture}
\begin{tikzpicture}[domain=0:4]
    \draw[->] (0,0) -- (4.2,0) node[midway,below]{b)};
    \draw[->] (0,0) -- (0,4.2);
    \draw[thick,color=cyan]  (0,1) -- (1,1);
    \draw[thick,color=cyan]  (1,1) -- (1,3); 
    \draw[thick,color=cyan]  (1,3) -- (3,3) node[midway,above] {$\Graph_x$};
    \draw[thick,color=cyan]  (3,3) -- (3,2); 
    \draw[thick,color=cyan]  (3,2) -- (4,2);
\end{tikzpicture}
\begin{tikzpicture}[domain=0:4]
    \draw[->] (0,0) -- (4.2,0) node[midway,below]{c)};
    \draw[->] (0,0) -- (0,4.2);
    \draw[thick,color=cyan]  (0,1) -- (1,1);
    \draw[thick,color=cyan]  (1,1) -- (1,3); 
    \draw[thick,color=cyan]  (1,3) -- (3,3);
    \draw[thick,color=cyan]  (3,3) -- (3,2); 
    \draw[thick,color=cyan]  (3,2) -- (4,2);
    
	\draw[color=brown,fill] (1,1.3) circle[radius=2pt] node[right]{P};
	\draw[color=brown,fill] (1,2.6) circle[radius=2pt] node[right]{Q};
	\draw[color=brown,fill] (2.7,3) circle[radius=2pt] node[above]{R};
	\draw[color=brown,fill] (3,2.5) circle[radius=2pt] node[right]{S};

    \end{tikzpicture}
\caption{In a), the graph of a piecewise constant path $x$ is plotted, b) shows the completion of its graph $\Graph_x$. In c), four points on the completed graph are marked to illustrate the order on the graph: For $A,B \in \Graph_x$, we have $A\preceq B$ iff  $A$ is reached before $B$ when $\Graph_x$ is traced out from left to right. In the situation depicted in $c)$, we find $P \prec Q \prec R \prec S$.} 
\end{figure}
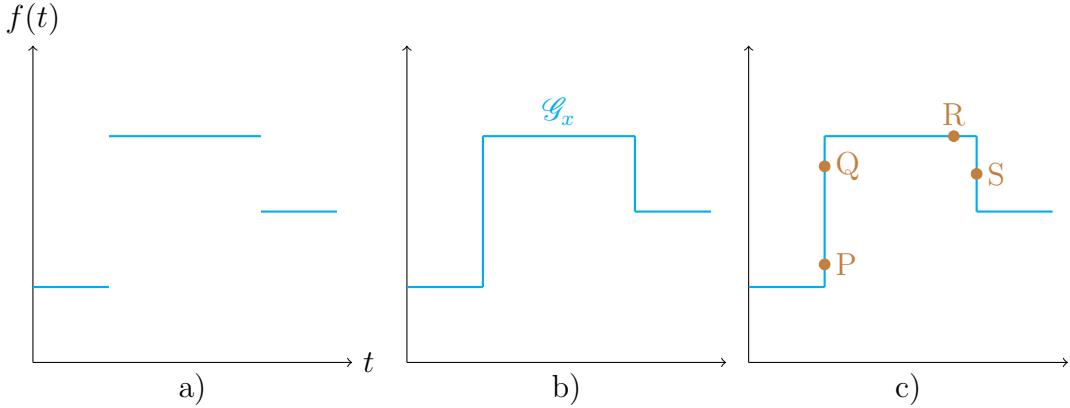

We then define a \emph{parametric representation} of $\Graph_x$ to be a continuous function $(r,u)$ that maps $[T_0,T]$ onto $\Graph_x$ such that $t \mapsto (r(t),u(t))$ is increasing with respect to ($\Graph_x,\preceq$). Denote the set of all parametric representations of $x$ as $R_x$.

For $x^1, x^2 \in D([T_0,T])$, we define the $M_1$-metric as 
\begin{align}\label{eqdef:M1metric}
d_{T}^{M_1}(x^1,x^2) = \inf_{\substack{(r^j, u^j) \in R_{x^j} \\ j=1,2}} \max(\|r^1-r^2\|_{\infty},\|u^1-u^2\|_{\infty}),
\end{align}
where $\|.\|_{\infty}$ is the supremum norm on $C([T_0,T])$. If we want to relate this to the picture with particles described earlier, we let the particles travel along the respective completed graphs and are allowed some freedom in choosing the velocities of the particles (albeit, due to the requirement that the parametric representations are increasing with respect to $(\Graph_x,\preceq)$, the velocities can never be negative). Then, two functions are close to each other in the $M_1$-topology, if there are velocity profiles for the particles such that the vertical and horizontal distance between the particles remains uniformly small. 

We provide some fundamental results regarding the $M_1$-topology which play a crucial role in many proofs of this paper.

The space $D([T_0,T])$ endowed with the $M_1$-topology is a Polish space, even though the metric defined in \eqref{eqdef:M1metric} is incomplete. It is generally not very pleasant to work directly with
the metric defined above, and we will seek to avoid doing so whenever possible. The following theorem will prove to be expedient in this endeavor, as it allows us to relate convergence in the $M_1$-topology to pointwise convergence when every path in the sequence is monotone. Let $M$ be the space defined in \eqref{eq:DistFunctiondef}.

\begin{lemma}\label{thm:M1monconvergence}
Let $(\ell^n)_{n\in \N},\ell \subset M$. Then $\ell^n\to\ell$ in the $M_1$-topology if and only if $\ell_t^n \to \ell_t$ for each $t$ in a subset of full Lebesgue measure of $[T_0,T]$ that includes $T_0$ and $T$.
\end{lemma}
\begin{proof}
See \cite[Theorem 12.5.1]{whitt2002stochastic}.
\end{proof}

Convergence in the uniform norm implies convergence in the $M_1$-topology. The next theorem shows that if the limit path is continuous,
the converse holds as well. 

\begin{lemma}\label{thm:M1contconvergence}
Suppose that  $x^n \to x$ in $D([T_0,T])$ equipped with the $M_1$-topology. Then 
we have locally uniform convergence at all continuity point of $x$. In particular,
for all points $t$ at which $x$ is continuous it holds that $x^n_t \to x_t$.
\end{lemma}
\begin{proof}
See \cite[Theorem 12.4.1]{whitt2002stochastic}.
\end{proof}

Similarly to the $J_1$-topology, the $M_1$-topology does not turn $D([T_0,T])$ into a topological vector space. In particular, addition is not continuous in general. However, the following result holds.
\begin{lemma}\label{lemma:M1addition}
Assume that $x^n \to x$ and $y^n \to y$ in $D([T_0,T])$ equipped with the $M_1$-topology. If $x$ and $y$ have no common jumps of opposite sign, that is  $$\Delta x_t  \cdot\Delta y_t \geq 0, \quad t \in [T_0,T],$$ then $x^n + y^n \to x+y$ in $D([T_0,T])$. 
\begin{proof}
See \cite[Theorem 12.7.3]{whitt2002stochastic}.
\end{proof}
\end{lemma}

If $(\ell_n)_{n\in \N}$ is a sequence of distribution functions such that $\ell_n \to \ell$ in $\DistFunctions$,  
Lemma~\ref{thm:M1monconvergence} tells us that $\ell_n \to \ell$ in the $M_1$-topology if and only if $\ell^n$ converges pointwise to
$\ell$ in the interval endpoints. To remove this restriction of convergence in the interval endpoints, we will consider processes
on all of $[0,\infty)$, and (somewhat artifically) extend the domain to $[-1,\infty)$, where we let the processes stay constant
for all negative times. Following \cite{whitt2002stochastic}, we need to define the $M_1$-metric on noncompact domains.\\

\begin{definition}\label{def:M1infty}
We say that $x^n \to x$ in $D([T_0,\infty))$ if $x^n$ converges to $x$ in $D([T_0,T_k])$ for each $T_k$ in some sequence $(T_k)_{k\in \N}$ with $T_k \to \infty$, where the sequence $(T_k)_{k\in \N}$ may depend on $x$. For $t > 0$, let $\hat d_t$ be a metric that makes $D([T_0,t])$ complete, then $D([T_0,\infty))$ is metrized by
\begin{align*}
d_{[T_0,\infty]}^{M1} (x^1, x^2) := \int_{T_0}^{\infty} e^{-t}(\hat d_t(x^1 , x^2) \wedge 1)~ \mathrm{d}t.
\end{align*}
\end{definition}
Equipped with this metric, $D([T_0,\infty))$ is a Polish space by construction.  We have the following equivalence.
\begin{lemma}\label{lemma:Dinftyconvergence}
For $x^n, x \in D([T_0,\infty))$, convergence of $x^n$ against $x$ with respect to $d_{[T_0,\infty]}^{M_1}$ is equivalent to convergence of $x^n$ against $x$ with respect to $\hat d_t$ for all $t > T_0$ where $x$ is continuous. 
\end{lemma}
\begin{proof}
See \cite[Theorem 12.9.3]{whitt2002stochastic}.
\end{proof}

\begin{lemma}\label{thm:Borelgeneval}
The Borel $\sigma$-field of $D([T_0,\infty))$ endowed with the $M_1$-topology is generated by the evaluation mappings.
\end{lemma}
\begin{proof}
By \cite[Theorem 1.14c]{jacod2013limit}, the claim holds for the Borel $\sigma$-field on $D([T_0,\infty))$ generated by the $J_1$-topology. 
By definition, the $J_1$-topology is stronger than the $M_1$-topology, and as any two comparable Lusin spaces have the same Borel sets \cite[p.101]{schwartz1973radon},
the claim follows.
\end{proof}

\begin{lemma}\label{lemma:limM1inf}
Assume that $x^n \to x$ in $D([T_0,\infty))$. Then if $t \in [T_0,\infty)$ is a continuity point of $x$, it holds that
$
\lim_{n\to\infty} \inf_{s\in [T_0,t]} x^n_s = \inf_{s \in [T_0,t]} x_s.
$
\end{lemma}
\begin{proof}
Fix a continuity point $t$ of $x$ and choose $T>t$ as a continuity point of $x$ as well. Then by Lemma~\ref{lemma:Dinftyconvergence} we have $x^n\to x$ in $D([T_0,T])$. By Theorem 13.4.1 in \cite{whitt2002stochastic}, the map 
$
x \mapsto \inf_{s \in [T_0,\cdot]}x_s
$
is continuous from $D([T_0,T])$ to $D([T_0,T])$, and therefore the claim follows from 
Lemma~\ref{thm:M1contconvergence}.
\end{proof}

\section{Proofs regarding physical solutions}\label{appB}
We start by introducing some useful notation.
\begin{definition}\label{def:lawstopped}
If $(X,\tau,\L)$ is a solution to the McKean--Vlasov problem \eqref{eq:mckeanproblem}, we denote by $\lawstopped_{t-}$ the marginal subprobability distribution at time $t-$ of the particles surviving up to time $t$, i.e.,
\begin{align*}
\lawstopped_{t-}(A) := \P{\tau \geq t, X_{t-} \in A}, \quad t \geq 0
\end{align*}
and we denote the measure corresponding to the minimal solution as
$\lawstoppedmin$.
\end{definition}

The next technical lemma is a key result when it comes to showing that
physical solutions of the particle system converge to physical solutions of
the McKean--Vlasov problem. Roughly speaking, it says that there is a very small chance
of observing more than one jump of a macroscopic proportion of particles in a small interval and
can be seen as an extension of \eqref{eq:physicalparticlejump} to small intervals. 
This result in its original form is due to \cite[Proposition 5.3]{delarue2015particle}. 
We follow the proof given in \cite[Lemma 3.10]{ledger2021mercy} here.

\begin{lemma}\label{lem:physparticleestimate}
Suppose that $\E{X_{0-}} < \infty$ and fix $T > 0$. Let $(\Xphys^N,\Lphys^N)$ be a physical solution to the particle system \eqref{eq:particleN} and let $\lawstopped_{t-}^{N}$ be the corresponding subprobabilty measure as defined in Definition~\ref{def:lawstoppedparticle}. Then there is a constant $C > 0$ 
 such that for every (sufficiently small) $\e > 0$ 
\begin{align*}
\P{\lawstopped_{t-}^{N}([0,\alpha z + C \e^{1/3}]) \geq z ~~ \forall z \leq \frac{1}{\alpha}(\Lphys_{t+\e}^{N} - \Lphys_{t-}^{N}) - C \e^{1/3}} \geq 1 - C\e^{1/3}, \quad t < T,
\end{align*}
whenever $N \geq \e^{-1/3}$.
\end{lemma}

\begin{proof}[Proof of Lemma~\ref{lem:physparticleestimate}]
Note that $\frac{N}{\alpha}(\Lphys_{t+\e}^{N} - \Lphys_{t-}^{N})$ equals the number 
of particles defaulting in the interval $[t,t+\e]$. By definition, as $\Lphys^{N}$ is a physical solution, if $t_0$ is any jump time in $[t,t+\e]$, we must have
\begin{align*}
\lawstopped_{t_0 -}^{N} \left(\left[0, \alpha \frac{k}{N}\right]\right) \geq \frac{k}{N} \quad \text{for} \quad k = 0,1,\dots, \frac{N}{\alpha}\Delta \Lphys_{t_0}^{N}.
\end{align*}
Let $t_1,\dots,t_m$ be the jump times of $\Lphys^N$ in $[t,t+\e]$. If 
$k \leq \frac{N}{\alpha} (\Lphys_{t+\e}^{N}-\Lphys_{t-}^{N})$, then there are numbers $k_1,\dots,k_m$ such that $k_i \leq \frac{N}{\alpha}\Delta \Lphys_{t_i}^{N}$ and $k_1 + \dots + k_m = k$, which then shows that 
\begin{align*}
\sum_{i=1}^{m} \lawstopped_{t_i-}^{N}\left(\left[0,\alpha \frac{k_i}{N}\right]\right) \geq \frac{1}{N}\sum_{i=1}^{m}{k_i} = \frac{k}{N}.
\end{align*}
Now by definition, the number on the left-hand side of the above
inequality is dominated by $\frac{1}{N} \sum_{i=1}^{N} \ind{E_1^{i,k}}$ where
\begin{align*}
E_1^{i,k} := \big\{ \Xphys_{t-}^{i,N} - \alpha \frac{k}{N} - \e(1+ \sup_{s\leq t+\e} |\Xphys_s^{i,N}|) - \sup_{h \leq \e} |B_{t+h}^i - B_{t}^i|\leq0,~ \tau^{i,N} \geq t \big\}.
\end{align*}
We have obtained that for $k = 0,1,\dots, \frac{N}{\alpha}(\Lphys_{t+\e}^{N} - \Lphys_{t-}^{N})$
\begin{align}\label{eq:kEeventestimate}
\frac{1}{N} \sum_{i=1}^{N} \ind{E_1^{i,k}} \geq \frac{k}{N}.
\end{align}
Now fix $z\in\R$ such that $z \leq \frac{1}{\alpha}(\Lphys_{t+\e}^{N} - \Lphys_{t-}^{N}) - 2 \e^{1/3}$ and set $k_0 := \lfloor z + 2\e ^{1/3} \rfloor \leq \frac{N}{\alpha}(\Lphys_{t+\e}^{N} - \Lphys_{t-}^{N}).$ This choice implies that $z \geq \frac{k_0}{N} - 2 \e^{1/3}$ as well as $\frac{k_0}{N} \geq z + 2 \e^{1/3} - \frac{1}{N}$ by definition, and we also have that \eqref{eq:kEeventestimate} holds for $k=k_0$. Define
\begin{align*}
E_2^i := \big\{ \e(1+ \sup_{s\leq t+\e} |\Xphys_s^{i,N}|) + \sup_{h \leq \e} |B_{t+h}^i - B_{t}^i| \geq \e^{1/3}\big\}.
\end{align*}
On the event $E_1^{i,k_0} \cap (E_2^{i})^{c},$ we have $\Xphys_{t-}^{i,N} - \alpha \frac{k_0}{N} \leq \e^{1/3},$ and hence
\begin{align*}
\Xphys_{t-}^{i,N} - \alpha z \leq \Xphys_{t-}^{i,N} - \alpha \left(\frac{k_0}{N} - 2 \e^{1/3}\right) \leq (1+2\alpha)\e^{1/3}.
\end{align*}
It follows that
\begin{align*}
\lawstopped_{t-}^{N}([0,\alpha z + (1+2\alpha)\e^{1/3}]) \geq \frac{1}{N} \sum_{i=1}^{N} \ind{E_1^{i,k_0}} \ind{(E_2^{i})^c}.
\end{align*}
Finally, letting $E:= \big\{\frac{1}{N} \sum_{i=1}^{N} \ind{E_2^i} \leq \e^{1/3}\big\}$, we find that on the event $E$ we have
\begin{align*}
\lawstopped_{t-}^{N}([0,\alpha z + (1+2\alpha)\e^{1/3}]) &\geq \frac{1}{N} \sum_{i=1}^{N} \ind{E_1^{i,k_0}} - \frac{1}{N}\sum_{i=1}^{N} \ind{E_2^i} \geq \frac{k_0}{N} - \e^{1/3} \\
&\geq (z + 2\e^{1/3} - {1}/{N}) - \e^{1/3} = z + \e^{1/3} - {1}/{N}.
\end{align*}
Taking $N \geq \e^{-1/3}$ implies $\e^{1/3} - 1/N \geq 0$, so we conclude that
\begin{align*}
\lawstopped_{t-}^{N}([0,\alpha z + (1+2\alpha)\e^{1/3}]) \geq z \quad \text{on}\quad E, 
\end{align*}
for any $z \leq \frac{1}{\alpha}(\Lphys_{t+\e}^N - \Lphys_{t-}^{N}) - 2 \e^{1/3}$. 
It remains to prove that there is a $C >0$  such that $\P{E^c} \leq C\e^{1/3}$. To do that, note that $\mathbb E\big[\sup_{h \leq \e} |B_{t+h}^i - B_{t}^i|\big] \leq \sqrt{{\pi}/{2}} \e$ and
\begin{gather*}
\mathbb E\big[1+ \sup_{s\leq T} |\Xphys_s^{i,N}|\big] \leq 1 + \E{X_{0-}} + \mathbb E\big[\sup_{s\leq T}|B_s|\big] + \alpha \leq  1 +  \alpha + \E{X_{0-}} + \sqrt{{\pi}/{2}}T.
\end{gather*}
Now letting $c \geq 1 + \alpha + \E{X_{0-}} + \sqrt{\frac{\pi}{2}}(1+T)$ we find, by applying the Markov inequality twice
\begin{align*}
\P{E^c} &= \mathbb P\Big(\frac{1}{N}\sum_{i=1}^{N} \ind{E_2^i} > \e^{1/3}\Big)
\leq \e^{-1/3} \frac{1}{N} \sum_{i=1}^{N} \P{E_2^i} \\
&\leq\e^{1/3} (\mathbb E\big[1+ \sup_{s\leq T} |\Xphys_s^{i,N}|\big] + \mathbb E\big[\sup_{h \leq \e} |B_{t+h}^i - B_{t}^i|\big]) \leq 2c\e^{1/3}.
\end{align*}
Now $C:= \max(2c,(1+2\alpha))$ satisfies the requirements of the lemma.
\end{proof}

\begin{proof}[Proof of Theorem~\ref{thm:physconvergetophys}]
Let $\EmpMphys_N$ be the empirical measure corresponding to $\Xphys^N$.
From  Corollary~\ref{cor:Phantomcorollary} we know that there 
are random variables $\EmpMextphys, \EmpMextphys_N$ such that, after passing to subsequences if necessary,
$
\law(\EmpMphys_N) = \law(\iota_{\alpha}(\EmpMextphys_N)),$
$\EmpMextphys_N \to \EmpMextphys,$ and $\iota_{\alpha}(\EmpMextphys_N) \to \iota_{\alpha}(\EmpMextphys)$
almost surely. Without loss of generality we may  assume that $\EmpMphys_N = \iota_{\alpha}(\EmpMextphys_N)$ and set $\EmpMphys= \iota_{\alpha}(\EmpMextphys)$.
Note that this implies $\Linfphys = \alpha\langle \EmpMphys, \ellfunc \rangle$ and thus, by Proposition~\ref{thm:finitedimconvergence}, that $\Linfphys$ solves the McKean--Vlasov problem \eqref{eq:mckeanproblem}. 
In the following three steps we prove that $\Linfphys$ is physical by verifying Condition~\eqref{eqdef:physicaljump}.\\

\underline{Step 1:} We show convergence of the laws of the subprobability measures $\lawstoppedphys_{t-}^N$. For $t \geq 0,$ let $\pi_{t-}(x)$ denote the map $x \mapsto x_{t-}$ and define
the transformation $S\colon D([-1,\infty)) \rightarrow \R$ through 
$S_{t-}(x) = \pi_{t-}(x)(1-\ellfunc_{t-}(x))$. 
Note that for $f \in C_b(\R)$ and $\mu \in \mathcal{P}(D([-1,\infty)))$ it holds that
\begin{align*}
\langle S_{t-}(\mu),f \rangle = \int_{}f(x_{t-}\ind{[\tau_0(x)\geq t]})~\mathrm{d}\mu(x) = \int_{}f(x_{t-})\ind{[\tau_0(x)\geq t]}~\mathrm{d}\mu(x) + f(0) \langle \mu, \ellfunc_{t-}\rangle.
\end{align*}
It follows that
\begin{align*}
\langle \lawstoppedphys_{t-}^N, f \rangle &= \langle S_{t-}(\EmpMphys_N), f \rangle - \frac{f(0)}{\alpha} \Lphys_t^N, \quad
\langle \lawstoppedphys_{t-}, f \rangle = \langle S_{t-}(\lawphys), f \rangle - \frac{f(0)}{\alpha} \Linfphys_t. 
\end{align*}
Let $J$ be the set of discontinuity points of $\Linfphys$. We claim that $S_{t-} \circ \iota_{\alpha}$ is continuous at $\EmpMextphys$-almost every $(w,\ell) \in \ExtendedE$ whenever $t \notin J$. To see this, consider that $\pi_{t-}$ is
continuous at all paths $x \in D([-1,\infty))$ such that $t$ is a continuity point of $x$ by Lemma~\ref{thm:M1contconvergence}. Secondly, suppose that for $x = \iota_{\alpha}(w,\ell)$, the measure $\delta_x$ satisfies the crossing property \eqref{eq:crossingmu}. Then, applying Lemma~\ref{lemma:muellconvergence} for $\muext=\delta_{(w,\ell)}$ and for a sequence $(\muext^n)_{n\in \N}$ converging to $\muext$ with $\muext^n = \delta_{(w^n,\ell^n)}$, it follows that $\ellfunc_t \circ \iota_{\alpha}$ is continuous at $(w,\ell)$ except possibly for $t = \tau_0(x)$. The same holds for $\ellfunc_{t-}\circ \iota_{\alpha}$. We have
\begin{align*}
\EmpMextphys(\{(w,\ell) \in \ExtendedE: \tau_0(\iota_{\alpha}(w,\ell)) = t\}) = \EmpMphys(\{x \in D([-1,\infty)) \colon \tau_0(x) = t\}) =
\P{\tauphys = t}
\end{align*}
and by definition, $\P{\tauphys = t} > 0$ if and only if $t \in J$. Since for $x = \iota_{\alpha}(w,\ell)$, the measure $\delta_x$ satisfies the crossing property for  $\EmpMextphys$-almost every $(w,\ell) \in \ExtendedE$, we have thus proved the aforementioned continuity property of $S_{t-} \circ \iota_{\alpha}$. It follows from the Portmanteau theorem that for every $t \notin J$ 
\begin{align*}
\lim_{N\to\infty}  S_{t-}(\EmpMphys_N) = 
\lim_{N\to\infty} (S_{t-} \circ \iota_{\alpha})(\EmpMextphys_N)
= (S_{t-} \circ \iota_{\alpha})(\EmpMextphys)
= S_{t-}(\EmpMphys),
\end{align*}
almost surely on $\mathcal{P}(\R)$. Therefore we obtain, for $t \notin J$,
$$
\lim_{N\to\infty} \lawstoppedphys_{t-}^N = \lawstoppedphys_{t-},
$$
almost surely on $\mathcal{S}(\R)$, where $\mathcal{S}(\R)$ is the space
of subprobability measures on $\R$ endowed with the topology of weak convergence. \\

\underline{Step 2:} Fix $T>0$ and a sufficiently small $\e>0$. We take the limit as $N\to\infty$ of 
\begin{align}\label{eq:inequalityprobproofphysmin}
\P{\lawstoppedphys_{t-}^{N}([0,\alpha z + C \e^{1/3}]) \geq z ~~ \forall z \leq \frac{1}{\alpha}(\Lphys_{t+\e}^{N} - \Lphys_{t-}^{N}) - C \e^{1/3}} \geq 1 - C\e^{1/3}, \quad \quad t < T.
\end{align}
The above equation holds due to Lemma~\ref{lem:physparticleestimate}, where we make use of the assumption $\E{X_{0-}}<\infty$.
Fix $t, t+\e \in [0,T)\setminus J$ and $z\in \R$ such that  $z < \frac{1}{\alpha}(\Linfphys_{t+\e} - \Linfphys_{t-}) - C\e^{1/3}$. Introduce the events
\begin{align*}
A_z^{N}:= \left\{z \leq \frac{1}{\alpha}(\Lphys_{t+\e}^{N} - \Lphys_{t-}^{N}) - C\e^{1/3}\right\}.
\end{align*}
By assumption, we know that $\Lphys_{t+\e}^{N} - \Lphys_{t-}^{N} \to \Linfphys_{t+\e} - \Linfphys_{t-}$, and therefore it holds that
$ \lim_{N\to\infty} \P{A_z^N} = 1.$
Recalling Step 1 and \eqref{eq:inequalityprobproofphysmin}, on applying the Portmanteau theorem and the reverse Fatou lemma we find
\begin{align*}
\lawstoppedphys_{t-}[0,\alpha z+ C\e^{1/3}] &\geq \mathbb E\big[\limsup_{N\to\infty}\lawstoppedphys_{t-}^{N}[0,\alpha z+ C\e^{1/3}]\big] \geq \limsup_{N\to\infty} \E{\lawstoppedphys_{t-}^{N}[0,\alpha z+ C\e^{1/3}]} \\
&\geq \limsup_{N\to\infty} \E{\lawstoppedphys_{t-}^{N}[0,\alpha z+ C\e^{1/3}]\ind{A_z^N}} \geq z(1-C\e^{1/3}).
\end{align*}
We have established
\begin{align*}
\lawstoppedphys_{t-}[0,\alpha z+ C\e^{1/3}] \geq z(1-C\e^{1/3}), \quad z < \frac{1}{\alpha}(\Linfphys_{t+\e} - \Linfphys_{t-}) - C\e^{1/3}, \quad t,t+\e \in [0,T)\setminus J.
\end{align*}

\underline{Step 3:} We take the limit as $\e \to 0.$ To that end, consider that for $f \in C_b(\R)$, by the dominated convergence theorem, the map
\begin{align*}
t \mapsto \int_{}f(x_{t-})\ind{[\tau_{0}(x)\geq t]}~\mathrm{d}\EmpMphys(x) = \langle \lawstoppedphys_{t-},f\rangle
\end{align*}
is left-continuous, and hence we have
$
\lim_{s\nearrow t} \lawstoppedphys_{s-} = \lawstoppedphys_{t-}
$
in $\mathcal{S}(\R)$. Now fix $t \in [0,T)$ and let $t_n, \e_n$ be such that $t_n, t_n + \e_n \notin J$ with $t_n < t < t_n + \e_n$, $t_n \nearrow t$ and $\e_n \searrow 0$. Let $0 \leq z <\frac{1}{\alpha}\Delta\Linfphys_t$, then as $\Linfphys$ is \cadlag, for all sufficiently large $n$ we have $z < \frac{1}{\alpha}(\Linfphys_{t_n + \e_n} - \Linfphys_{t_n -}) - C\e_n^{1/3}.$ Then, for any $\ee > 0$, by  the Portmanteau theorem and Step~2, we obtain
\begin{align*}
\lawstoppedphys_{t-}[0,\alpha z + \ee] &\geq \limsup_{n\to\infty} \lawstoppedphys_{t_n-}[0,\alpha z + \ee]
\geq \limsup_{n\to\infty}\lawstoppedphys_{t_n-}[0,\alpha z + C\e_{n}^{1/3}]  \\
&\geq \limsup_{n\to\infty} z(1-C\e_n^{1/3})
=  z.
\end{align*} 
Letting $\ee \to 0$, as $T>0$ was arbitrary we finally see that
$
\lawstoppedphys_{t-}[0,\alpha z] \geq z$ for each $z < \frac{1}{\alpha}\Delta \Linfphys_t$ and $ t \geq 0,
$
which implies
\begin{align*}
\Delta \Linfphys_t \leq \alpha \inf\{z > 0 \colon \lawstoppedphys_{t-}[0,\alpha z] < z\}.
\end{align*}
The reverse inequality follows from Proposition~\ref{thm:physjumpinequality}.
\end{proof}

\section{Supplements}\label{appC}
\begin{example} \label{ex1}
Define $$ \Omega := C([0,\infty))^{\mathbb{N}}$$
and let $\omega$ be the the canonical process on $C([0,\infty))^{\N}$. We take $\mathcal{F}$ to be the product $\sigma$-field induced by the Borel $\sigma$-field on $C([0,\infty))$. Let $\mathbb{P} = \bigotimes_{n\in\N} \eta$ for a probability measure $\eta$ on $C([0,\infty))$,  such that the law of
$\omega^i$ under $\eta$ coincides with the law of $X_{0-}+B$
 for a Brownian motion $B$ and a random variable $X_{0-}$.
 We write $E:= D([-1,\infty))$ in the following.

Let $\mathcal{R}$ be the space of all random variables $\mu \colon \Omega \rightarrow \mathcal{P}(E)$ and let $\mu \in \mathcal{R}$. As the Borel $\sigma$-field in $\mathcal{P}(E)$ is generated by the mappings $\mu \mapsto \mu(A)$ where $A$ is a Borel set in $E$ (cf. \cite[Proposition 5.7]{carmona2018probabilisticI}), any such map is measurable. In addition, Lemma~\ref{lemma:limM1inf} shows that the map $x \mapsto \inf_{s\leq \cdot} x_s$ is continuous on $E$, and as the evaluation mappings $\pi_t := x \mapsto x_t$ are measurable, it follows that the set
$$ A_t:= \big\{x \in E\colon \pi_t \big(\inf_{s\leq \cdot} x_s\big) \in (-\infty,0]\big\}$$
is a Borel set in $E$. We find that
$ \langle \mu, \ellfunc_t \rangle = \mu(A_t),$
so we obtain that $\omega \mapsto \langle \mu(\omega), \ellfunc_t \rangle$ is measurable for every $t \geq 0$. As the Borel $\sigma$-field on $E$ is generated by the evaluation mappings, 
it follows that $\omega \mapsto \langle \mu(\omega), \ellfunc\rangle$ is measurable as a map into $E$. 

By Lemma~\ref{lemma:M1addition}, for $N \in \N$ the map 
\begin{align*}
\Psi^N &\colon E \times C([0,\infty))^{\N} \rightarrow \mathcal{P}(E) \\
     & (x,\omega) \mapsto \frac{1}{N} \sum_{i=1}^{N} \delta_{\omega^i - \alpha x} 
\end{align*}
is continuous. It follows that 
$$\Gamma_N[\langle \mu, \ellfunc\rangle](\omega) = \langle\Psi^N(\langle \mu(\omega), \ellfunc\rangle,\omega),\ellfunc\rangle
=\Psi^N(\langle \mu(\omega), \ellfunc\rangle,\omega)(A_t)$$ and is thus measurable. This shows that the map $\omega \mapsto c_N(\mu(\omega))$ is measurable for every $\mu \in \mathcal{R}$. Defining $$\mu_N^{(1)} := \Psi^N(0,\omega), \quad \mu_{N}^{(k)} := \Psi^N (\langle \mu_{N}^{(k-1)}(\omega), \ellfunc\rangle, \omega), \quad k \in \N,$$ by Lemma~\ref{lem:miniteration} we find that $\EmpM_N = \mu_{N}^{(N)} \in \mathcal{R}$.

\end{example}

\bibliographystyle{plainnat}

\end{document}